\documentclass[bj,numbers]{imsart}

\usepackage{caption}
\usepackage{graphicx}

\RequirePackage[numbers]{natbib}

\startlocaldefs

\newtheorem{theorem}{Theorem}[section]
\newtheorem{lemma}[theorem]{Lemma}

\theoremstyle{definition}
\newtheorem{definition}[theorem]{Definition}
\newtheorem{example}[theorem]{Example}

\theoremstyle{remark}
\newtheorem{remark}[theorem]{Remark}

\endlocaldefs

\begin{document}

\begin{frontmatter}
\title{Optimal Markovian coupling for finite activity L\'{e}vy processes}
\runtitle{Optimal Markovian coupling for finite activity L\'{e}vy processes}

\begin{aug}
\author[A]{\inits{W.S.}\fnms{Wilfrid S.}~\snm{Kendall}\ead[label=e1]{w.s.kendall@warwick.ac.uk}}
\author[B]{\inits{M.B.}\fnms{Mateusz B.}~\snm{Majka}\ead[label=e2]{m.majka@hw.ac.uk}}
\author[C]{\inits{A.}\fnms{Aleksandar}~\snm{Mijatovi\'{c}}\ead[label=e3]{a.mijatovic@warwick.ac.uk}}
\address[A]{Department of Statistics,
	University of Warwick, Coventry, CV4 7AL\printead[presep={,\ }]{e1}}

\address[B]{School of Mathematical and Computer Sciences, Heriot-Watt University, Edinburgh, EH14 4AS, UK\printead[presep={,\ }]{e2}}

\address[C]{Department of Statistics,
	University of Warwick, Coventry, CV4 7AL\printead[presep={,\ }]{e3}}
\end{aug}

\begin{abstract}
We study optimal Markovian couplings of Markov processes, 
where the optimality is understood in terms of minimization 
of concave transport costs 
between the time-marginal distributions of the coupled processes.
We provide explicit constructions of such optimal couplings for
one-dimensional finite-activity L\'{e}vy processes 
(continuous-time random walks)
whose jump distributions are unimodal but not necessarily symmetric. 
Remarkably, the optimal Markovian coupling does not depend on the specific concave transport cost. 
To this end, we combine McCann's results on optimal transport and Rogers' results on random walks
with a novel uniformization construction 
that allows us to characterize all Markovian couplings of finite-activity L\'{e}vy processes.
In particular, we show that the optimal Markovian coupling for 
finite-activity L\'{e}vy processes with non-symmetric unimodal L\'{e}vy measures 
has to allow for non-simultaneous jumps of the two coupled processes.
\end{abstract}


\end{frontmatter}

 \section{Introduction}\label{sec:introduction}
This paper considers Markovian 
couplings of a finite activity L{\'e}vy process
(``continuous-time random walk'') \(X\), started at \(0\) and \(a>0\) respectively.
We show that if the jump distribution is unimodal
then there is a unique Markovian coupling \(Z\) with \(Z_0=X_0+a\),
based on the anti-monotonic rearrangement construction described in
\cite{McCann-1999}, such that 
for any $t >0$, the expected cost
\(\mathbb{E}\left[\phi(X_t-Z_t)\right]\)
is simultaneously minimized for all bounded concave increasing cost functions \(\phi:[0,\infty)\to[0,\infty)\). 
If the jumps are symmetric then this optimal Markovian coupling can be viewed as a 
reflection coupling. 
However, our arguments apply to all unimodal non-symmetric jump distributions, even as extreme as exponential (where all jumps are non-negative). In the latter case the optimal Markovian coupling involves the processes jumping independently till crossover,
when they couple.

Recall that a \emph{coupling} of two probability measures $\mu_1$ and $\mu_2$, defined respectively 
on Polish measure spaces $(E_1, \mathcal{E}_1)$ and $(E_2, \mathcal{E}_2)$, is a 
probability measure $\mu$ defined on the product space $(E_1 \times E_2, 
\mathcal{E}_1 \otimes \mathcal{E}_2)$ with marginals $\mu_1$ and $\mu_2$. If we 
consider two canonical random variables $X$ and $Y$ on $(E_1, \mathcal{E}_1, 
\mu_1)$ and $(E_2, \mathcal{E}_2, \mu_2)$, respectively, then any coupling $\mu$ 
of $\mu_1$ and $\mu_2$ specifies a law of a random vector $(X',Y')$ on $(E_1 
\times E_2, \mathcal{E}_1 \otimes \mathcal{E}_2)$, which we also call a coupling 
of the random variables $X$ and $Y$. These definitions extend naturally to any 
finite number of measures or random variables \cite[Chapter 3, Section 3]{Thorisson-2000}. 
Furthermore, by considering random variables on $(E^I, 
\mathcal{E}^{\otimes I})$ for a time index set $I$, we can consider couplings of 
$(E, \mathcal{E})$-valued stochastic processes indexed by $I$ \cite[Chapter 4]{Thorisson-2000}.

The question ``what is
the best possible coupling?''
clearly depends on 
context and the chosen notion of optimality. In the context of 
couplings of stochastic processes, the most widespread criterion of optimality 
uses the notion of a \emph{maximal coupling}, which is a coupling that minimizes the 
meeting time of the two coupled processes started from different initial 
positions  \cite{Griffeath-1975, Pitman-1976, Goldstein-1978, 
BanerjeeKendall-2014, HsuSturm-2013, Bottcher-2017, JackaMijatovicSiraj-2014, 
JackaMijatovic-2015, ErnstKendallRobertsRosenthal-2019}. Through the classical 
coupling inequality \cite{Lindvall-1992}, the problem of minimizing the meeting 
time is closely related to the problem of obtaining sharp upper bounds on the 
total-variation distance between time-marginal distributions of the coupled 
processes.

In recent years couplings have found numerous novel applications to the 
problem of quantifying convergence rates of Markov processes in other 
probabilistic metrics, in particular the Wasserstein distances. This has been an 
active area of research in stochastic analysis \cite{Eberle-2016, 
EberleGuillinZimmer-2019, Majka-2017, Majka-2019, LuoWang-2019}, computational 
statistics \cite{EberleMajka-2019, MajkaMijatovicSzpruch-2020} and machine 
learning \cite{Cheng-2018, Cheng-2020}. These applications focus on 
different optimality criteria for couplings, related to the optimal transport 
interpretation of the Wasserstein distances \cite{Chen-1994, 
GangboMcCann-1996, LiangSchillingWang-2020}. To see this, recall that the 
$L^p$-Wasserstein distance is defined for any $p \geq 1$, and for any two 
probability measures $\mu_1$ and $\mu_2$ on a measure space $(E, \mathcal{E})$, 
by
\begin{equation}\label{eq:defWasserstein}
W_{\rho,p}(\mu_1,\mu_2) 
\quad=\quad
\left( \inf_{\pi \in \Pi(\mu_1,\mu_2)} \int_{E \times E} \rho(x,y)^p \pi (\operatorname{d}{x} \, \operatorname{d}{y}) \right)^{1/p} \,,
\end{equation} 
where $\rho$ is a metric on $(E, \mathcal{E})$ and $\Pi(\mu_1,\mu_2)$ is the set 
of all couplings of $\mu_1$ and $\mu_2$. By \eqref{eq:defWasserstein},
any coupling $(X_t, Y_t)_{t \geq 0}$ of two stochastic processes 
$(X_t)_{t \geq 0}$ and $(Y_t)_{t \geq 0}$ on $(E, \mathcal{E})$ provides an 
upper bound on the Wasserstein distance between their time-marginal 
distributions $\mathcal{L}(X_t)$ and $\mathcal{L}(Y_t)$ for any $t > 0$:
\begin{equation*}
W_{\rho,p}(\mathcal{L}(X_t),\mathcal{L}(Y_t)) 
\quad\leq\quad
\left( \mathbb{E} \left[ \rho(X_t, Y_t)^p\right] \right)^{1/p} \,.
\end{equation*}
Sharp upper bounds on such Wasserstein distances
arise by minimizing the quantity $\mathbb{E} \left[ \rho(X_t, 
Y_t)^p\right]$ over all processes $(X_t, Y_t)_{t \geq 0}$ from a given class of 
couplings. 
General results guarantee existence of the optimal 
coupling for lower semi-continuous metrics $\rho$ 
(\cite[Theorem 2.19]{Kellerer-1984} or \cite[Chapter 4]{Villani-2009});
however these results are 
non-constructive and hence 
do not yield
quantitative 
bounds on $W_{\rho,p}$. 
In this paper we 
discuss explicit constructions of
optimal 
couplings, 
forcing
us to consider a more specific setting and in 
particular to focus on a specific class of couplings. 
A natural requirement for 
the choice of such a class is that both the coupled processes
should respect 
an appropriate common
filtration. 
Definition \ref{defImmersion} below
formalizes this by introducing 
the notions of \emph{immersion couplings} and \emph{Markovian couplings}. 
We focus on a specific 
class of metrics $\rho(x,y) = \phi(|x-y|)$ based on continuous concave 
cost functions $\phi: \mathbb{R} \to \mathbb{R}$. 
In particular 
our results
cover 
metrics of the form $\rho(x,y) = |x-y|^p$ for $p \in (0,1)$.

The task of finding optimal immersion couplings is non-trivial, 
even for discrete-time processes, 
(in which case a general coupling can be specified 
by 
iteratively
coupling the transition (jump) distributions of two processes step-by-step). 
Rogers \cite{Rogers-1999} considered the problem of identifying 
the 
one-step optimal
coupling for one-dimensional random walks with unimodal jump 
distributions, and characterised this in terms of a joint 
cumulative distribution function:
he then showed the resulting random walk coupling 
was optimal 
(amongst Markovian couplings)
over all time
and for a wide range of bounded convex pay-off functions. 
The present paper
uses a different characterisation to generalize \cite{Rogers-1999} mildly to the case of 
one-dimensional Markov chains with non-symmetric unimodal 
jump 
 distributions perhaps with atoms at the modes. 
 This generalisation is useful when moving to continuous-time processes (Section \ref{sec:Levy}). 

For more general Markov 
processes with symmetric transition kernels, it seems to be established folklore 
that applying a symmetrized version of McCann's \cite{McCann-1999} anti-monotonic rearrangement (AMR) 
coupling is the preferred choice for concave costs, even in 
multi-dimensional isotropic settings \cite{Eberle-2016, EberleMajka-2019, 
Majka-2017, Cheng-2018}. However, optimality
has been established only for (one-dimensional) probability measures
\cite{McCann-1999}, not in the context of general
stochastic processes. 
If the transition kernels are not symmetric then the general case 
becomes very complicated and no general optimality results have been 
considered. Without imposing any geometric conditions on transition 
probabilities, the only natural solution for a step-by-step coupling of 
stochastic processes seems to be to preserve the common mass of their jump 
measures and to couple the remaining mass independently: the 
\emph{basic coupling} \cite[Example 2.10]{Chen-2005}. Luo and Wang
\cite{LuoWang-2019} proposed 
the \emph{refined basic coupling},
preserving only half of the common mass of the jump measures, and coupling 
the remainder 
(the non-common mass) 
synchronously rather than independently, but imposing a 
different transformation on the second half of the common mass. 
This 
construction is typically
non-optimal but permits sharp upper bounds on the Wasserstein 
distances between time-marginals of various stochastic processes 
\cite{LiangWang-2020, LiangSchillingWang-2020, LiangMajkaWang-2019, 
HuangMajkaWang-2022} without needing geometric conditions on 
transition distributions. Determining the optimal 
coupling in general seems too ambitious a goal for now. 

Even in settings where an optimal coupling is available for given probability 
measures, constructing an optimal coupling for stochastic processes having such 
measures as their jump distributions can be highly non-trivial, especially in 
continuous-time. A significant part of the problem is the need to work with 
general characterisations of all possible Markovian couplings of stochastic 
processes from a given class. Such characterisations have been used in coupling 
diffusions or birth-death processes \cite{Chen-1994b}, but have not yet been 
exploited for general pure jump L\'{e}vy processes
\cite[discussion in Section 4]{LiangSchillingWang-2020}.

The present paper discusses explicit 
constructions of optimal Markovian couplings 
in the case of one-dimensional discrete-time 
Markov chains and continuous-time finite-activity L\'{e}vy processes whose jump 
distributions are unimodal (but not necessarily symmetric). To this end, 
Section~\ref{sec:rw} revisits results from \cite{McCann-1999} and 
\cite{Rogers-1999} to provide a new construction of McCann's AMR coupling, via 
the Skorokhod representation. Methods from \cite{Rogers-1999} 
establish optimality for AMR, first for couplings of individual random variables and then for Markov chains (Theorem \ref{thm:optimality}). 
Section \ref{sec:Levy} then introduces a novel uniformization construction for 
immersion couplings of finite activity L\'{e}vy processes (Theorem 
\ref{thm:uniformization}), permitting application of modified 
arguments from Section \ref{sec:rw} in the case of continuous time. The final 
result, Theorem \ref{thm:LevyOptimality}, constructs an optimal 
coupling of finite-activity L\'{e}vy processes with a unimodal L\'{e}vy measure 
$\nu$ (i.e., with jumps distributed according to $\nu/\nu(\mathbb{R})$), 
using step-by-step application of the AMR coupling to the processes' 
counterparts with jump measures $\frac{1}{2\nu(\mathbb{R})}\nu + 
\frac{1}{2}\delta_0$. This final observation leads to a conclusion that, for 
optimal couplings of finite-activity L\'{e}vy processes with non-symmetric 
unimodal L\'{e}vy measures, some non-zero jumps of the first coupled process 
may
correspond to zero jumps of the second, and vice versa; so optimally coupled
processes need not jump simultaneously.

\section{Optimal Markovian coupling for Markov chains}\label{sec:rw}

As indicated in the introduction, we need to consider couplings which simultaneously respect a common
filtration to which both random processes are adapted, but such that the coupling construction does not ``look into the future'' \cite[Section 1]{BurdzyKendall-2000}.  
We use the expressive and concise language of martingales:
\begin{definition}[Immersion and Markovian couplings for random processes]\label{defImmersion}
	Consider two random processes $X$ and $Y$ taking values in a Polish 
	space \((E,\mathcal{E})\),
	defined on the same probability space \((\Omega,\mathcal{F},\mathbb{P})\) 
	and adapted to 
	a common filtration of \(\sigma\)-algebras \(\{\mathcal{F}_t:t\geq0\}\).
	We suppose 
	here and throughout that
	the filtration is
	regularized, in the sense that
	\(\mathcal{F}_t=\bigcap_{s>t}\mathcal{F}_s\)
	and moreover
	for all \(t\geq0\) the \(\sigma\)-algebra \(\mathcal{F}_t\) contains all \(\mathbb{P}\)-null sets. 
	We say that $(X,Y)$
	is an \emph{immersion coupling} (synonyms: 
	\emph{faithful coupling} \cite{Rosenthal-1997},
	\emph{co-adapted coupling} 
	 \cite{ConnorJacka-2008,Connor-2013}) if the following holds:
	All martingales for the natural filtration of $X$ are also martingales in 
	the common filtration \(\{\mathcal{F}_t:t\geq0\}\),
	similarly all martingales for the natural filtration of $Y$ are also 
	martingales in the common filtration \(\{\mathcal{F}_t:t\geq0\}\).
	If additionally \((X,Y)\) is Markovian in the common filtration
	then we say \((X,Y)\) is a \emph{Markovian coupling}. 
\end{definition}

Strictly we should speak of Markovian \emph{immersion} couplings: however the abbreviated term ``Markovian coupling'' is now established in the literature. 
For optimal couplings the distinction between Markovian and immersion is not major:
optimal immersion couplings will often be Markovian by virtue of maximizing
the optimality criterion at each step of time (see the proof of Theorem \ref{thm:optimality} below).
For further discussion 
see 
\cite{Kendall-2013a, BanerjeeKendall-2018}. 

When considering optimality , the advantage of restricting to
the classes of immersion or Markovian couplings, 
since these typically 
permit explicit constructions
using tools such as stochastic 
calculus.
For this reason such couplings are currently widely used in the stochastic 
analysis / machine learning literature when investigating convergence rates of 
Markov processes.
In the following
we seek couplings which are optimal among the class of immersion or Markovian couplings.

We work with the following notion.
\begin{definition}[Optimal (immersion) Coupling]
	Consider an immersion coupling of two Markov processes
	$X$ and $Y$ taking values in a Polish space \((E,\mathcal{E})\). We say that this is an
	\emph{optimal (immersion) coupling}
	(for a specified loss function \(\phi:E\times E\to[0,\infty)\) vanishing on the diagonal)
	if
	\[
	\mathbb{E}\left[ {\phi(X(t+s),Y(t+s)) \;\Big|\; X(s), Y(s)} \right]
	\]
	is minimized amongst all immersion couplings
	of $X$ and $Y$
	for all \(t, s\geq 0\) and all possible \(X(s)\), \(Y(s)\).
In a mild abuse of terminology, if the minimizing coupling is Markovian
we call it an \emph{optimal Markovian coupling}.
\end{definition}
Generally the optimal coupling depends greatly on the specific function $\phi$ being considered.
We consider the special case
for
which the Polish state space is the real line \(\mathbb{R}\),
and we require that optimality holds simultaneously for \emph{all choices}
of \(\phi(x,y)=\phi(|x-y|)\) for concave continuous \(\phi:[0,\infty)\to[0,\infty)\) vanishing
at zero.
Note that \cite{McCann-1999} considers this type of simultaneous convex-cost optimality for couplings of probability measures. 
Note also that $\phi(x,y)=|x-y|^p$ is a rather unsatisfactory cost when $p \geq 1$: 
if $p = 1$ then too many couplings are optimal \cite[Theorem 3.8]{Chen-1994}; 
for $p > 1$ the optimal coupling is synchronous
[\citealp[Thm~3.7]{GangboMcCann-1996}; \citealp[Thm~1.1]{JackaMijatovic-2015}]. 
The concave case ($0 < p < 1$) is more interesting.  

The concept of an anti-monotonic rearrangement coupling of probability measures
\cite{McCann-1999} 
is fundamental for our optimality proofs. 
A useful tool for this is the Skorokhod representation.

\subsection{Skorokhod representation}\label{sec:skorokhod}
Recall that the Skorokhod representation
generates a copy of 
a given scalar random variable \(X\) using
its distribution function \(F(x)=\mathbb{P}[X\leq x]\)
and a Uniform\((0,1)\) random variable \(U\).
We use the fact that distribution functions are themselves c\`{a}dl\`{a}g 
(\underline{c}ontinue \underline{\`{a}} \underline{d}roite, \underline{l}imite \underline{\`{a}} \underline{g}auche), 
i.e., they are right-continuous with left limits. 
The functional inverse of \(F\) need not exist; 
adequate (though non-unique) substitute inverses can be defined as follows
\cite[Section 3.12, for example]{Williams-1991}:
\begin{align}
F^{-1,+}(z) \quad&=\quad \inf\{x: F(x) >   z\}
\quad =\quad \sup\{x: F(x)\leq z\}\,,
\label{eq:skorokhod}\\
F^{-1,-}(z) \quad&=\quad \inf\{x: F(x)\geq z\} 
\quad =\quad \sup\{x: F(x) <   z\}\,.
\nonumber
\end{align}
Set \(X^+=F^{-1,+}(U)\) and \(X^-=F^{-1,-}(U)\).
Then \(\mathbb{P}[X^+\neq X^-]=0\) and
\(\mathbb{P}[X^-\leq x]=\mathbb{P}[X^+\leq x]=F(x)\).

Consequently a copy of \(X\) can be generated using 
a substitute inverse of 
the cumulative distribution function \(F\) as above,
and it does not matter which of the possible 
substitute inverses is employed.
Indeed, if $F$ is continuous then
direct arguments show that
$F(X^+)=F(X^-)=U$.

\subsection{McCann's anti-monotonic rearrangement (AMR) coupling}\label{sec:mccann}
McCann \cite{McCann-1999} discusses an \emph{anti-monotonic rearrangement coupling} (AMR coupling)
between two distribution functions \(F_1\)
and \(F_2\) defined
on the real line. 
If $X$ and $Y$ are the random variables 
implementing the coupling, then
(conditional on \(X\neq Y\))
the random variables \(X\) and \(Y\) are
antithetically
coupled,
in the sense that \(X\) and \(Y\) are
negatively associated
(the informal notion of
``antithetic coupling'' corresponds to the notion of ``antithetic simulation''
\cite{HammersleyMorton-1956}).

We consider only the particular case, in which \(F_2\) 
stochastically dominates \(F_1\):
\[
F_1(x) \quad\geq\quad F_2(x) \qquad \text{for all }x \in \mathbb{R}\,.
\]
This stochastic domination is related to a condition of McCann \cite[Proposition 2.12]{McCann-1999}, 
where one of two measures is required to vanish on a connected set, while the 
other vanishes on its complement, after removing their common mass (the 
difference is that McCann works on the circle, while we work on the line).  
We further require that
for each value \(\ell\geq0\) the super-level set
\begin{equation}\label{eq:superlevel}
S_\ell\quad=\quad\{x\in\mathbb{R}: F_1(x)-F_2(x)\geq\ell\}
\end{equation}
should be either connected (hence an interval) or empty: 
elementary properties of distribution functions
show that the interval
\(S_\ell\) is non-increasing in \(\ell\) 
and is bounded for positive \(\ell\).
Note that the connectedness assumption will be satisfied for our main example, 
where $F_1$ corresponds to a (weakly) unimodal probability distribution, and 
$F_2$ is a translation of $F_1$, cf.\ Section \ref{sec:unimodal}.  
\noindent
It follows that we can find \(\zeta\) such that 
\(\zeta\in S_\ell\)
for all \(\ell<\ell_\text{max}=\sup_{x \in \mathbb{R}}(F_1(x)-F_2(x))\).
Indeed, since \(S_\ell\) is connected and bounded,
and must be non-empty if \(\ell<\ell_\text{max}\),
removing \( S_\ell\) from \(\mathbb{R}\) leaves a complement consisting of 
 two (disjoint)
connected components:
\(\mathbb{R}\setminus S_\ell\
=C^1_\ell\sqcup C^2_\ell\).
Both components, being connected, are half-infinite (closed or open)
intervals. 
We define disjoint half-infinite intervals 
\({C^1}\) and \({C^2}\) by
\({C^1}=
\cup_{\ell< \ell_\text{max}} C^1_\ell\)
and
\({C^2}=
\cup_{\ell<
	\ell_\text{max}} C^2_\ell\).
Choose \(\zeta\) to be the right end-point of
\({C^1}\): 
note that \({C^1}\subseteq(-\infty,\zeta]\)
and \({C^2}\subseteq[\zeta,\infty)\).
However, we note that equality cannot hold for both these inclusions,
and may not hold in either.

The interval \(S_\ell\) 
is non-increasing in \(\ell\),
so
the difference \(F_1(x)-F_2(x)\)
must be non-decreasing when \(x<\zeta\) 
and 
non-increasing when \(x\geq\zeta\). 
Note that \(F_1\) and \(F_2\) are c\'{a}dl\'{a}g (so we can include
\(x=\zeta\) in the non-increasing range),
but need not be continuous, so we cannot assume that 
\(F_1(\zeta)-F_2(\zeta)=\ell_\text{max}\).

It is now convenient to switch to the language of probability
measures. 
Recall that if \(\mu_1\) and \(\mu_2\) are two non-negative measures then their
\emph{meet}
is defined by
\[
(\mu_1 \wedge\mu_2)(A)\quad=\quad \sup\{\mu_1(B)\wedge\mu_2(B)\;:\; B\subseteq A, 
B\in\mathcal{E}\}\,.
\]
Consider 
the Hahn-Jordan decomposition for the signed measure
\(\mu_1-\mu_2=p \nu_1-p \nu_2\), where \(\mu_1\) is the probability 
measure corresponding to \(F_1\) and \(\mu_2\) is the probability 
measure corresponding to \(F_2\).
Here $1-p$ is the total mass of $\mu_1\wedge\mu_2$, while
\begin{equation*}
p\nu_1 \quad=\quad \mu_1-(\mu_1\wedge\mu_2)\,, \qquad \text{ and } \qquad 
p\nu_2 \quad=\quad \mu_2-(\mu_1\wedge\mu_2)\,.
\end{equation*}
Thus \(p\in[0,1]\) while \(\nu_1\) and \(\nu_2\) 
are probability measures of disjoint support. 
The stochastic domination together with
connectedness of 
the super-level sets \eqref{eq:superlevel} implies
that the non-negative measures \(p \nu_1\) and \(p \nu_2\) 
of the Hahn-Jordan decomposition are supported on the
disjoint half-infinite intervals \(C^1\) and \(C^2\).

Define a third probability measure \(\nu^*\) by \((1-p)\nu^*=\mu_2\wedge\mu_1=\mu_2-p\nu_2=\mu_1-p\nu_1\).
By stochastic domination,
\[
(1-p)\nu^*((-\infty,x])\quad=\quad F_1(x) - F_2(x)
\qquad\text{ for all }x \in \mathbb{R} \,.
\]

McCann's 
AMR coupling
is a maximal coupling of 
random variables \(X\) and \(Y\) with distribution functions
\(F_1\) and \(F_2\),
implementing
the representation implied by the Hahn-Jordan decomposition
\begin{align*}
\mu_1 \quad&=\quad (1-p)\nu^* + p\nu_1\,,
\\
\mu_2 \quad&=\quad (1-p)\nu^* + p\nu_2
\end{align*}
in an anti-monotonic manner,
using a single Uniform\((0,1)\) random variable \(U\)
as follows:
\begin{enumerate}
	\item If \(U\geq p\) then use the Skorokhod representation
	(and the notation of \eqref{eq:skorokhod}) to generate 
	\begin{equation}\label{eq:skorokhod-def1}
	X \quad=\quad Y \quad=\quad(G^*)^{-1,+}((U-p)/(1-p))    
	\end{equation}
	where
	\((G^*)^{-1,+}\) is the substitute inverse of the distribution function \(G^*\) determined by
	\(G^*(x)=\nu^*((-\infty,x])\) for all \(x \in \mathbb{R} \);
	\item If on the other hand \(U<p\)
	then use the Skorokhod representation twice over
	to generate
	\begin{align}\label{eq:skorokhod-construction}
	X\quad&=\quad(G_1)^{-1,+}(U/p)\,,
	\\\nonumber
	Y\quad&=\quad(G_2)^{-1,+}((p-U)/p)\,,
	\end{align}
	where
	\((G_i)^{-1,+}\) for $i \in \{ 1, 2 \}$ are the substitute inverses of the distribution functions \(G_i\) 
	which are determined by
	\(G_i(x)=\nu_i((-\infty,x])\) for all \(x \in \mathbb{R} \);
\end{enumerate}
Note that instead of \((G^*)^{-1,+}\), $(G_1)^{-1,+}$ and $(G_2)^{-1,+}$, we 
could have used \((G^*)^{-1,-}\), $(G_1)^{-1,-}$ and $(G_2)^{-1,-}$, 
respectively. Hence the AMR coupling depends on the choice of the 
substitute inverse used in the Skorokhod representation, as explained in Section 
\ref{sec:skorokhod}. 
If \(X\neq Y\) then
\eqref{eq:skorokhod-construction} delivers
an anti-monotonic relationship between
\(X\) and \(Y\) \emph{via} \(U\): as \(U\)
increases it is the case that \(X\) is
non-decreasing and \(Y\) is non-increasing.
However in general we cannot infer that \(Y\)
is an anti-monotonic \emph{function} of \(X\);
if \(G_1\) or \(G_2\) have jumps then this
may not be the case.

In the case where
\(F_1\) and \(F_2\)
have densities (\(f_1\) and \(f_2\) respectively)
then McCann's AMR coupling
admits a more direct description.
In that case 
\(F_1\) and \(F_2\) are continuous, while
\(C^1_\ell\) and
\(C^2_\ell\) are half-infinite closed intervals 
with \(\zeta\) defined as above.
Moreover
we have \({C^1}=(-\infty,\zeta)\)
and \({C^2}\subseteq(\zeta,\infty)\),
while \(F_1(\zeta)-F_2(\zeta)=\ell_\text{max}\)
by continuity of \(F_1\) and \(F_2\).

In this case of densities, the coupled copy \(Y\) 
may then be constructed in terms of \(X\) and a different 
auxiliary random variable \(U\), 
with \(U\) being Uniform\((0,1)\) independent of \(X\).
Set
\[
Y \quad=\quad
\begin{cases}
\rho_{F_1,F_2}(X) & \text{ when }U\geq(f_2(x)\wedge f_1(x))/f_1(x)\,,
\\
X            & \text{ when }U < (f_2(x)\wedge f_1(x))/f_1(x)\,.
\end{cases}
\]
Here \(\rho_{F_1,F_2}: C^1 \to C^2 \) is the 
\emph{anti-monotonic rearrangement function},
defined for \(x\in {C^1}\) by 
\begin{equation}\label{eq:AMR-density}
\rho_{F_1,F_2}(x)
\quad=\quad
\inf\left\{\rho\;:\;
\int_{\rho}^\infty (f_2(t)-f_1(t)) \operatorname{d}{t}
=
\int_{-\infty}^x (f_1(t)-f_2(t)) \operatorname{d}{t}
\right\}\,.
\end{equation}
Since \(F_1\) and \(F_2\) are continuous,
it follows that 
\begin{equation}\label{eq:defZeta}
\zeta \in 
\{x \in \mathbb{R}\;:\;  F_1(x)-F_2(x)=\ell_\text{max} \} \neq 
\emptyset
\,.
\end{equation}
Moreover,
if \(\zeta\in\{x:F_1(x)-F_2(x) = \ell_\text{max}\}\)
then from \eqref{eq:AMR-density} we may deduce that
\[
\int_{-\infty}^\zeta (f_1(t)-f_2(t)) \operatorname{d}{t}
\;=\;
F_1(\zeta)-F_2(\zeta)
\;=\; 
(1-F_2(\zeta))-(1-F_1(\zeta))
\;=\;
\int_\zeta^\infty (f_2(t)-f_1(t)) \operatorname{d}{t}\,.
\]
Furthermore, \(\int_u^\zeta (f_2(t)-f_1(t)) \operatorname{d}{t}\)
is always negative when \(u<\zeta\), since then
\(u\in C^1_\ell\) for some \(\ell<\ell_\text{max}\).
Therefore \(\zeta=\rho_{F_1,F_2}(\zeta)\).

\begin{figure}
	\includegraphics{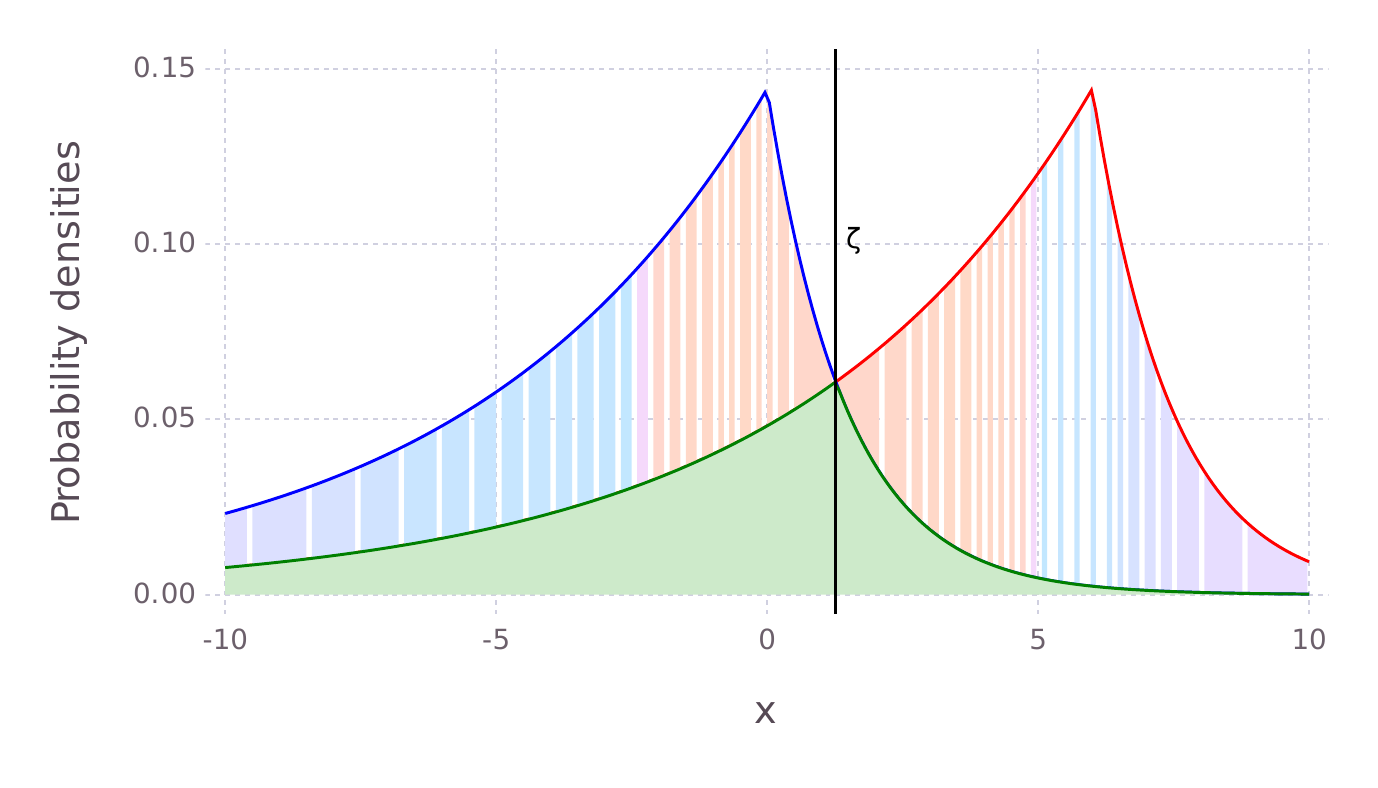}
	\caption{Anti-monotonic rearrangement. 
		In the two uncoupled regions (left and right) the  banding indicates the use of 
		 anti-monotonic rearrangement to map vertical strips in the left region to vertical strips in the right region and \emph{vice versa}.
		\label{fig:AMR-1}}
\end{figure}

\subsection{Unimodality and the measure-majorization of shifts}\label{sec:unimodal}

We plan to apply AMR to produce optimal Markovian couplings of two copies of a 
Markov chain $(X_k)_{k=0}^{\infty}$ started from two different initial points. 
To this end, we will need to couple transition probabilities of 
$(X_k)_{k=0}^{\infty}$ with their respective counterparts shifted by $a$ for any 
$a > 0$. Note that this corresponds to coupling jump distributions of two 
processes with different starting points (separated by $a > 0$). 
To prepare for this,
we 
discuss a general definition of (weak) unimodality which makes no explicit reference to probability densities.

\begin{theorem}\label{thm:unimodal-prob}
	Suppose a distribution function \(F\) 
	corresponds to a random variable \(X\)
	such that for all \(x>0\) there exists \(a(x)\in(0,x)\) such that
	\begin{align*}
	\mathbb{P}[X+x \in E]
	\quad&\geq\quad 
	\mathbb{P}[X \in E] \qquad\text{ for measurable } E\subset [a(x),\infty)\,,
	\\
	\mathbb{P}[X-x \in E]
	\quad&\geq\quad 
	\mathbb{P}[X \in E] \qquad\text{ for measurable } E\subset (-\infty,-a(x)]\,.
	\end{align*} 
	Then the probability law corresponding to \(F\)
	can be expressed
	as a mixture of (a) a probability density \(f\) which is (weakly) unimodal
	at \(0\) and (b) a Dirac mass at $0$.
\end{theorem}
Here ``weakly unimodal'' at $y \in \mathbb{R}$ means that there is a version of 
the density \(f\) which is non-decreasing on \((-\infty,y)\)
and non-increasing on \((y,\infty)\).

It is convenient to prove a measure-theoretic version of Theorem \ref{thm:unimodal-prob}. 
Consider two measures $\mu_1$ and $\mu_2$ defined on a $\sigma$-algebra $\mathcal{A}$.
For any measurable $G \subset \mathbb{R}$, 
the notation $\mu_1|_G \geq \mu_2|_G$ means
that $\mu_1(A) \geq \mu_2(A)$ for any $A \in \mathcal{A}$ such that $A \subset G$. 
For $x \in \mathbb{R}$  
let $\delta_x$ denote the Dirac mass at $x$, so that the convolution measure $\delta_x * \nu$ is the shift of $\nu$ by $x$.
\begin{theorem}\label{thm:unimodal-measure}
	Suppose a non-negative measure \(\nu\) on $\mathbb{R}$ 
	satisfies the following majorization relationship:
	for all \(x>0\) there exists \(a(x)\in (0, x)\)
	such that 
	\begin{align}\label{eq:unimodal-measure}
	(\delta_x * \nu)|_{[a(x),\infty)} 
	\quad&\geq\quad \nu|_{[a(x),\infty)}\,,
	\\\nonumber
	(\delta_{-x} * \nu)|_{(-\infty,-a(x)]} 
	\quad&\geq\quad \nu|_{(-\infty,-a(x)]} \,.
	\end{align}
	Then \(\nu\) 
	has a density \(f\) on \(\mathbb{R}\setminus\{0\}\),
	which can be chosen to be
	c\`{a}dl\`{a}g non-decreasing on \((-\infty,0)\)
	and c\`{a}gl\`{a}d non-increasing on \((0,\infty)\). We say that (the density component of) $\nu$ is weakly unimodal.
\end{theorem}
Thus the majorization condition 
\eqref{eq:unimodal-measure} actually
corresponds to weak unimodality of the density at \(0\). 
For weak unimodality at an arbitrary $y \in \mathbb{R}$, one 
would need to consider $a(x) \in (y,y+x)$ for $x > 0$, \emph{etc}.

\begin{proof}
It suffices to treat the case of \((0,\infty)\):
	the case of \((-\infty, 0)\) then follows 
	by applying the symmetry of
	reflection in \(0\). Firstly we establish existence of a density for $\nu|_{(0,\infty)}$.
	
	Fix \(\varepsilon > 0\) and a set \(A \subset [\varepsilon, \infty)\). 
	Then choose a bounded interval \((0,H) \subset \mathbb{R}\) 
	of positive Lebesgue measure
	such that \(A \subset [a(x), \infty)\) whenever
	\(x \in (0,H)\). 
	We can do this since \(0 < a(x) < x \to 0\) as 
	\(x \downarrow 0\). 
	
	Now choose a Uniform\((0,H)\) random variable \(U\).
	Define a measure \(\mu\) by
	\(\mu(B) = \mathbb{E}[(\delta_U*\nu)(B)]\).
	Then \(\mu\) has a density with respect to  
	Lebesgue measure. 
	For the density
	of \(U\) by \(h=H^{-1}\mathbf{1}_{(0,H)}\), so
	\begin{equation*}
	\begin{split}
	\mu(B) \quad&=\quad \int h(x)\nu(B-x)\operatorname{d}{x}
	\;=\;
	\int h(x) \int_{B-x} \nu(\operatorname{d}{y}) \operatorname{d}{x}
	\;=\;
	\int \int h(x) \mathbf{1}_{B-x}(y) \nu(\operatorname{d}{y})  \operatorname{d}{x} 
	\\
	&=\quad
	\int \int h(x) \mathbf{1}_{B-y}(x) \operatorname{d}{x} \, \nu(\operatorname{d}{y}) 
	\;=\;
	\int \int_{B-y} h(x) \operatorname{d}{x} \, \nu(\operatorname{d}{y}) 
	\,.
	\end{split}
	\end{equation*}
	Hence \(\mu(B) = 0\) 
	for any Lebesgue null-set \(B \subset\mathbb{R}\).
	Thus \(\mu\) is an absolutely continuous measure.

	Now observe
	that for any \(x \in (0,H)\) we have
	\((\delta_x *\nu) |A \geq \nu|A\)
	(since \(A \subset [a(x), \infty)\)).
	Hence
	 almost surely
	\((\delta_U*\nu) |A \geq \nu|A\).
	This implies that
	\(\mu|A \geq \nu|A\) so long as \(A \subset [a(x), \infty)\),
	hence that \(\nu\) has a density
	on \(A\).
	This holds for all \(\varepsilon>0\) and
	\(A\subset [\varepsilon, \infty)\);
	therefore \(\nu\) 
	must have a density on all of \((0,\infty)\).
	
	The
	unimodality
	condition \eqref{eq:unimodal-measure}
	is
	symmetric under reflection, 
	so \(\nu\) has a density $f$ on all of \(\mathbb{R} \setminus \{0\}\).
	
	We now establish a weakly decreasing property of (a c\`{a}gl\`{a}d version of) the density of $\nu|_{(0,\infty)}$. 

Fix $z > 0$ and argue from condition \eqref{eq:unimodal-measure} that
		\[
		f(y-z) \quad\geq\quad f(y) \qquad \text{ for almost all } y > z >0\,.
		\]
		Indeed, if $f(y-z) < f(y)$ for $y \in A \subset (z, \infty)$ with $\operatorname{Leb}(A) > 0$, then
		\[
		\left( \delta_z * \nu \right) (A) \quad=\quad \int_A f(y-z) \operatorname{d}{y} \quad<\quad \int_A f(y) \operatorname{d}{y} = \nu(A) \,, 
		\]
	which contradicts \eqref{eq:unimodal-measure}. 
	So
	\(f(u) \geq f(x)\) for almost all $u$, $x$ with 
	\(u<x\in(0,\infty)\),
	and (for almost all \(x>0\))
	\begin{equation*}
	f(x) \quad\leq\quad 
	\bar{f}(x)
	 \quad=\quad 
	\operatorname{ess~inf}\{f(u)\;:\; 0<u<x\}\,.
	\end{equation*}
	Evidently $\bar{f}(x)$ is non-increasing in $x$ for $x > 0$.
	The proof is completed if we can
	show that \(f(x)=\bar{f}(x)\)
	for almost all \(x>0\), as we can then deduce that \(\bar{f}\)  is a 
	c\`{a}gl\`{a}d 
	non-increasing version of \(f\) on \((0,\infty)\). 
	This final step is a little more delicate.
	
	Consider a range \((a,a+T]\subset (0,\infty)\),
	and fix \(\varepsilon>0\).
	Consider the set \(M_{a,T,\varepsilon}\subset(a,a+T]\) given by
	\[
	M_{a,T,\varepsilon} = \{x\in(a,a+T]: f(x) < \bar{f}(x)-\varepsilon\}\,,
	\]
	and suppose \(\operatorname{Leb}(M_{a,T,\varepsilon})>0\). 
	Divide \((a,a+T]\) into \(K\) disjoint half-open intervals 
	of equal length \(T/K\) and suppose that
	\(L\) of these intersect with \(M_{a,T,\varepsilon}\)
	in a set of positive measure; evidently 
	\(L \times T/K \geq \operatorname{Leb}(M_{a,T,\varepsilon})\).
	For such an interval \(I=(x,y]\)
	we know that \(\operatorname{Leb}(I\cap M_{a,T,\varepsilon})>0\)
	and therefore (by the properties of \(\operatorname{ess~inf}\) and the monotonicity of $\bar{f}$)
	\[
	\bar{f}(y)\quad \leq\quad \bar{f}(x-)-\varepsilon\,,
	\]
	where $ \bar{f}(x-)=\lim_{z\uparrow x} \bar{f}(z)$.
	Hence \(\bar{f}\) must decrease by at least \(\varepsilon\) on
	such an interval \(I\).
	
	Summing over all intervals in the 
	above decomposition of \((a,a+T]\),
	we deduce
	\[
	\bar{f}(a-)-\bar{f}(a+T)
	\quad\geq\quad
	L \varepsilon \quad\geq\quad 
	\operatorname{Leb}(M_{a,T,\varepsilon}) (K/T) \varepsilon
	\,.
	\]
	But the left-hand side is bounded above by \(\bar{f}(a-)\),
	so we may deduce
	\[
	\operatorname{Leb}(M_{a,T,\varepsilon}) \quad\leq\quad
	\frac{\bar{f}(a-) }{\varepsilon}
	\frac{T}{K}
	\,.
	\]
	Since \(K\) can be chosen to be arbitrarily large,
	we deduce that \(\operatorname{Leb}(M_{a,T,\varepsilon})=0\). 
	This holds for all \(\varepsilon>0\) and all \(a,T>0\).
	Hence
	\(f(x)=\bar{f}(x)\) for almost all \(x>0\),
	as required to complete the proof.
	
\end{proof}

\subsection{Optimal Wasserstein coupling for concave costs}\label{subsubsect:AMRcoupling}

In order to prove optimality for suitable AMR couplings we apply
the results of Section \ref{sec:unimodal} to
probability distributions $F$ and $G$ on $\mathbb{R}$ such that $F$ is a mixture 
of a Dirac mass at $0$ with an absolutely continuous distribution whose density 
is weakly unimodal at $0$, while $G$ is a right-shift of a (possibly different) 
distribution of the same form.

We further suppose that there exists $\zeta \in \mathbb{R}$ with the property 
that the function $F - G$ is strictly increasing on $(-\infty,\zeta]$ and 
strictly decreasing on $[\zeta,\infty)$, cf.\ \eqref{eq:defZeta}. This is 
automatically true if $F$ is strictly unimodal and $G$ is a shift of $F$ to the 
right. 

As discussed at the beginning of Section \ref{sec:rw}, we are interested in 
finding couplings that minimize concave cost functions. However 
in this subsection 
it is convenient to
focus on the equivalent problem of 
maximizing convex payoffs (see the discussion before  Lemma 
\ref{lem:AMRoptimality2} for further comments on this equivalence). To this end, 
we apply the methods from \cite{Rogers-1999} to McCann's AMR coupling 
defined in Section \ref{sec:mccann}. Note that \cite{Rogers-1999} considers a 
convex-monotone coupling defined in a completely different way, see (2) therein. 
However,  \cite[Lemma 1]{Rogers-1999} combines with the optimality result 
below (which implies uniqueness of the optimal coupling by Remark 
\ref{remarkUniqueness}) to show that the AMR coupling is the same as the 
convex-monotone coupling if the distribution $F$ is strictly 
unimodal and $G$ is a right-shift of $F$. In general the two definitions are very 
different: 
AMR can be interpreted in many different ways in non-unimodal situations \cite{McCann-1999}
while
the convex-monotone coupling is well-defined in such cases.

By the definition of the AMR coupling $(X,Y)$ from Section \ref{sec:mccann}, for any $a < b \leq \zeta$
\begin{equation}\label{defAMRpart1}
\mathbb{P}\left[(X,Y) \in (a,b]^2 , X = Y \right] 
\quad=\quad \mathbb{P}\left[X \in (a,b] , X = Y \right] 
\quad=\quad G(b) - G(a) 
\end{equation}
and for any $\zeta < a < b$
\begin{equation}\label{defAMRpart2}
\mathbb{P}\left[(X,Y) \in (a,b]^2 , X = Y \right] 
\quad=\quad \mathbb{P}\left[Y \in (a,b] , X = Y \right] 
\quad=\quad F(b) - F(a) \,.
\end{equation}
Thus for any $a \leq \zeta < b$
\begin{equation*}
\mathbb{P}\left[(X,Y) \in (a,b]^2 , X = Y \right] 
\quad=\quad G(\zeta) - G(a) + F(b) - F(\zeta) \,.  
\end{equation*}

Now $F$ and $G$ are not necessarily diffuse, so \eqref{eq:AMR-density}
cannot be applied to define a mapping $\rho_{F,G}(x)$ for all $x \in \mathbb{R}$.
However, for any $x < \zeta$ such that $F$ does not have an atom at $x$ (i.e., 
for any $x \neq 0$ in the setting of this section), we can define a real number 
$\underline{\operatorname{AMR}}(x) \geq \zeta$ by
\begin{equation*}
\underline{\operatorname{AMR}}(x) 
\quad=\quad
\inf \{ y > \zeta : F(x) - G(x) \geq F(y) - G(y) \} \,.
\end{equation*}
So for any $a$ such that $a \leq \zeta < b$ and $F$ does not have an 
atom at $a$, if $a < X < Y$ and $b > \underline{\operatorname{AMR}}(a)$, 
then $Y \leq b$. 
Thus if $a \leq \zeta < b$
and
$b > \underline{\operatorname{AMR}}(a)$ then
\begin{equation}\label{defAMRpart3}
\begin{split}
\mathbb{P}\left[(X,Y) \in (a,b]^2 , X < Y \right]
\quad&=\quad
\mathbb{P}\left[a < X \leq \zeta < Y \leq b \right]
\quad=\quad 
\mathbb{P}\left[a < X \leq \zeta, X < Y \right] \\
&=\quad (F(\zeta) - F(a)) - (G(\zeta) - G(a))  \,. 
\end{split}
\end{equation}
Similarly, for any $a$ such that $a \leq \zeta < b$ and $F$ does not have an atom at $a$,
if $b \leq \underline{\operatorname{AMR}}(a)$ then
\begin{equation}\label{defAMRpart4}
\begin{split}
\mathbb{P}\left[(X,Y) \in (a,b]^2 , X < Y \right]
\quad&=\quad
\mathbb{P}\left[a < X \leq \zeta < Y \leq b \right]
\quad=\quad
\mathbb{P}\left[\zeta < Y \leq b, X < Y \right] \\
&=\quad (G(b) - G(\zeta)) - (F(b) - F(\zeta)) \,.
\end{split}
\end{equation}
Note that the corresponding probabilities are 
zero
in all the other possible cases of $a$ and $b$, so long as $a$ and $b$ respectively
avoid the atoms of $F$ and $G$.

We need the following technical result, 
whose proof is adapted from that of Lemma 1 in \cite{Rogers-1999}.

\begin{lemma}
	For any random variables $X$ and $Y$ and for any $c > 0$ we have
	\begin{equation*}
	\mathbb{E}\left[\left( c - |X - Y| \right)^{+}\right]
	\quad=\quad
	\int_{-\infty}^{\infty} \mathbb{P} \left[(X,Y) \in (t-c,t]^2 \right] \operatorname{d}{t} \,.
	\end{equation*}
	\begin{proof}
This follows by direct manipulation of the integral:
\begin{equation*}
		\begin{split}
& \int_{-\infty}^{\infty} \mathbb{P} \left[(X,Y) \in (t-c,t]^2 \right] \operatorname{d}{t} 
\quad=\quad 
\int_{-\infty}^{\infty} \big(\mathbb{P} \left[t - c < X \leq Y \leq t \right) + \mathbb{P} \left( t - c < Y < X \leq t\right] \big)
\operatorname{d}{t} \\
    &\qquad=\quad \mathbb{E} \left[ \int_{-\infty}^{\infty}
		\big( \mathbf{1}_{\{ Y \leq t < X + c , X \leq Y \}} + \mathbf{1}_{\{ X \leq t < Y + c , Y < X \}} \big) 
		\operatorname{d}{t} \right] \\
    &\qquad=\quad \mathbb{E} \left[ \left( X + c - Y \right)\mathbf{1}_{\{ X \leq Y \}}
                    + \left( Y + c - X \right) \mathbf{1}_{\{ Y < X \}} \right] 
        \quad=\quad \mathbb{E} \left[\left( c - |X - Y| \right)^{+}\right] \,.
		\end{split}
		\end{equation*}
	\end{proof}
\end{lemma}

As an immediate corollary, it follows that
\begin{equation*}
\mathbb{E}\left[\left( c - |X - Y| \right)^{+} \right]
\quad\leq\quad
\int_{-\infty}^{\infty} \big( \mathbb{P}\left[t - c < X \leq t \right] \wedge \mathbb{P}\left[t - c < Y \leq t \right] \big)
\operatorname{d}{t} \,. 
\end{equation*}
We will now show that the upper bound is actually obtained by the coupling by 
anti-monotone rearrangement defined in Section \ref{sec:mccann}, i.e., we will 
show that, since the AMR coupling $(X,Y)$ satisfies 
\eqref{defAMRpart1}-\eqref{defAMRpart4}, for Lebesgue-almost all $t$, then in this particular case
\begin{equation}\label{AMRisOptimal}
\mathbb{P} \left[(X,Y) \in (t-c,t]^2 \right]
\quad=\quad
\min \{ F(t) - F(t-c), G(t) - G(t-c) \} \,.
\end{equation}
In the sequel, condition \eqref{AMRisOptimal} is needed only for 
Lebesgue-almost all $t \in \mathbb{R}$ since we are interested in values of the 
expressions in \eqref{AMRisOptimal} after integration with respect to $t$. 
So 
it suffices to show that \eqref{AMRisOptimal} holds only for $t$ and $c$ such that 
$F$ does not have an atom at $t-c$ (so that \eqref{defAMRpart3} and 
\eqref{defAMRpart4} apply for $a = t - c$).

Consider the case of $t - c < \zeta \leq t$. Then we have
\begin{equation*}
\begin{split}
& \mathbb{P} \left[ (X,Y) \in (t-c,t]^2 \right]
\quad=\quad 
\\
& \qquad
\mathbb{P} \left[t - c < X \leq \zeta, X = Y \right] + \mathbb{P} \left[\zeta < Y \leq t, X = Y \right] 
+ \mathbb{P} \left[t - c < X \leq \zeta < Y \leq t \right] \,.
\end{split}
\end{equation*}
Applying \eqref{defAMRpart1} and \eqref{defAMRpart2}, the first two summands on the right add up to
\begin{equation*}
G(\zeta) - G(t-c) + F(t) - F(\zeta) \,.
\end{equation*}
If $t > \underline{\operatorname{AMR}}(t-c)$, then we can apply \eqref{defAMRpart3} to show that the third summand equals
\begin{equation*}
F(\zeta) - G(\zeta) + G(t-c) - F(t-c) \,,
\end{equation*}
while if $t \leq \underline{\operatorname{AMR}}(t-c)$ then by \eqref{defAMRpart4} the third summand is equal to
\begin{equation*}
F(\zeta) - G(\zeta) + G(t) - F(t) \,.
\end{equation*}
Note that if $t \leq \underline{\operatorname{AMR}}(t-c)$ then $F(t) - F(t-c) \geq G(t) - G(t-c)$, while if $t > \underline{\operatorname{AMR}}(t-c)$ then $F(t) - F(t-c) \leq G(t) - G(t-c)$, and hence addition yields \eqref{AMRisOptimal}. The other cases can be checked more directly, as the equivalent of the third summand is then zero.

We have proved the following result:
\begin{lemma}\label{lem:AMRoptimality}
	Let $F$ and $G$ be cumulative distribution functions. Suppose that there 
exists $\zeta \in \mathbb{R}$ with the property that the function $F - G$ is 
strictly increasing on $(-\infty,\zeta]$ and strictly decreasing on 
$[\zeta,\infty)$. Then for any $c > 0$, the AMR coupling $(X,Y)$ is the optimal 
coupling in the sense of maximizing the payoff function $\mathbb{E}\left[ (c - 
|X - Y|)^{+} \right]$. 
\end{lemma}

\begin{remark}\label{remarkUniqueness}
	From the argument above, it is evident that any coupling $(X,Y)$ satisfying 
\eqref{AMRisOptimal} must also satisfy \eqref{defAMRpart1}-\eqref{defAMRpart4} 
and, as a consequence, must be defined via 
\eqref{eq:skorokhod-def1}-\eqref{eq:skorokhod-construction}. This shows that the 
AMR coupling given by  
\eqref{eq:skorokhod-def1}-\eqref{eq:skorokhod-construction} is essentially the 
unique coupling $(X,Y)$ maximizing the payoff function $\mathbb{E}\left[ (c - |X 
- Y|)^{+} \right]$. This has to be understood in the sense that any coupling 
maximizing $\mathbb{E}\left[ (c - |X - Y|)^{+} \right]$ must satisfy 
\eqref{eq:skorokhod-def1}-\eqref{eq:skorokhod-construction}, however, condition 
\eqref{eq:skorokhod-def1}-\eqref{eq:skorokhod-construction} itself is only 
unique up to the choice of the Skorokhod representation as explained in Section 
\ref{sec:skorokhod}. In particular, for $F$ and $G$ with atoms, there will be 
more than one ``version'' of AMR satisfying 
\eqref{eq:skorokhod-def1}-\eqref{eq:skorokhod-construction}.
\end{remark}

From now on we will focus on the special case of $G = F_a$ for $a > 0$, where we 
define $F_a(x) = F(x - a)$ for all $x \in \mathbb{R}$. Notice that in this case 
the AMR coupling $(X,Y)$ of random variables with distributions $F$ and $F_a$ 
depends only on $X$ and $a$ and hence we will denote it by $(X,\operatorname{AMR}_a(X))$.

In order to proceed, we need the following lemma. Its proof is a reformulation 
of the proof of Lemma~2 in \cite{Rogers-1999}, and is included here for 
completeness.

\begin{lemma}\label{lemmaSupIsConvex}
	Fix a random variable $X$ whose distribution is weakly unimodal.
	For any $a$, $c > 0$ consider the function \(\psi\) defined by
	\begin{equation}\label{defPsi}
	\psi(a,c) \quad=\quad
	\sup_{Y: \mathcal{L}(Y) = \mathcal{L}(X + a)} \mathbb{E} \left[ \left( c - |X-Y| \right)^{+} \right] \,,
	\end{equation}
	where $\mathcal{L}(Y)$ is the law of the random variable $Y$. Then it is the case that
	\begin{equation*}
	\psi(a,c) + a \quad=\quad \psi(c,a) + c \,. 
	\end{equation*}
	\begin{proof}
		We know that
		\begin{equation*}
		\begin{split}
		\psi(a,c) \quad&=\quad
		\mathbb{E} \left[ \left( c - |X-\operatorname{AMR}_a(X)| \right)^{+} \right] 
		\quad=\quad \int_{-\infty}^{\infty} (F(t) - F(t-c)) \wedge (F_a(t) - F_a(t-c)) \operatorname{d}{t} \\
		\quad&=\quad \int_{-\infty}^{\infty} (F(t) - F(t-c)) \wedge (F(t-a) - F(t-c-a)) \operatorname{d}{t} \,.
		\end{split}
		\end{equation*}
As a consequence of the weak unimodality of the law of $X$, if the absolutely continuous part 
of $F$ has a density $f$, then for a given $c > 0$ there exists a constant 
$\zeta$ such that $f(\zeta) = f(\zeta-c)$ and the function $t \mapsto F(t) - 
F(t-c)$ increases up to $\zeta$ and decreases afterwards. Hence, for 
Lebesgue-almost all $a$ and $c > 0$, there exists $t_0 > \zeta$ such that 
\begin{equation}\label{t0condition}
		F(t_0) - F(t_0 - c) \quad=\quad F(t_0 - a) - F(t_0 - a - c) \,.
		\end{equation}
		and, for all $t \leq t_0$,
		\begin{equation*}
		(F(t) - F(t-c)) \wedge (F(t-a) - F(t-c-a)) \quad=\quad F(t-a) - F(t-c-a) \,,
		\end{equation*}
		whereas, for all $t > t_0$,
		\begin{equation*}
		(F(t) - F(t-c)) \wedge (F(t-a) - F(t-c-a)) \quad=\quad F(t) - F(t-c) \,.
		\end{equation*}
		Note that \eqref{t0condition} may fail for some combinations of $a$ and 
$c$ since $F$ is allowed to have an atom. However, as noted above,
it is sufficient if \eqref{t0condition} holds 
for Lebesgue-almost all $a$ and $c$. Furthermore, observe that condition 
\eqref{t0condition} is symmetric in $a$ and $c$ and hence if the 
roles of $a$ and $c$ in \eqref{defPsi} are exchanged then the same $t_0 = t_0(a,c)$ is obtained. 
Letting $\bar{F}(t) = 1 - F(t)$ be the complementary distribution function,
		\begin{equation*}
		\begin{split}
		\psi(a,c) \quad&=\quad
		\int_{-\infty}^{\infty} (F(t) - F(t-c)) \wedge (F(t-a) - F(t-c-a)) \operatorname{d}{t} \\
		&=\quad
		\int_{\infty}^{t_0} (F(t-a) - F(t-c-a)) \operatorname{d}{t} + \int_{t_0}^{\infty} (F(t) - F(t-c)) \operatorname{d}{t} \\
		&=\quad
		\int_{t_0 - c - a}^{t_0 - a} F(t) \operatorname{d}{t} + \int_{t_0 - c}^{t_0} \bar{F}(t) \operatorname{d}{t} \\
		&=\quad
		\int_{t_0 - c - a}^{t_0 - c} F(t) \operatorname{d}{t} + \int_{t_0 - c}^{t_0 - a} F(t) \operatorname{d}{t} + \int_{t_0 - c}^{t_0 - a} \bar{F}(t) \operatorname{d}{t} + \int_{t_0 - a}^{t_0} \bar{F}(t) \operatorname{d}{t} \\
		&=\quad \int_{t_0 - c - a}^{t_0 - c} F(t) \operatorname{d}{t} + c - a + \int_{t_0 - a}^{t_0} \bar{F}(t) \operatorname{d}{t}
		\quad=\quad \psi(c,a) + c - a \,,
		\end{split}
		\end{equation*}
		where the fourth equality is true when $a < c$, which can be assumed
without loss of generality, and the last equality follows from the fact that 
$t_0(a,c) = t_0(c,a)$ as observed above.
	\end{proof}
\end{lemma}

As a consequence 
\cite[Lemma 2 and the following discussion]{Rogers-1999} the function
\begin{equation*}
a \quad\mapsto\quad \sup_{Y: \mathcal{L}(Y) =\mathcal{L}(X + a)} \mathbb{E} \left[ \left( c - |X-Y| \right)^{+} \right]
\end{equation*}
is convex. 
 However a more general result can be obtained:
	\begin{lemma}\label{lem:generalConvexity}
		Let $\varphi: \mathbb{R}_+ \to \mathbb{R}_+$ be a bounded continuous convex and non-increasing function. Then
		\begin{equation*}
		a \quad\mapsto\quad
		\sup_{Y: \mathcal{L}(Y) =\mathcal{L}(X + a)} \mathbb{E} \left[ \varphi(|X-Y|) \right]
		\end{equation*}
		is bounded, continuous, convex, and non-increasing.
\begin{proof}
			First observe that any such function $\varphi$ can be represented as
			\begin{equation*}
			\varphi(|x-y|) \quad=\quad \lim_{n \to \infty} \sum_{i=1}^{n} \lambda_i^n (c_i^n - |x-y|)^{+} \,,
			\end{equation*}
			with positive constants $\lambda_i^n$, $c_i^n$ chosen to ensure that 
the sequence of functions $f_n(x,y) = \sum_{i=1}^{n} \lambda_i^n (c_i^n - 
|x-y|)^{+}$ is increasing. Hence
			\begin{equation}\label{eq:mct}
			\mathbb{E} \left[ \varphi(|X-Y|) \right] \quad=\quad
			\lim_{n \to \infty} \sum_{i=1}^{n} \mathbb{E} \left[ \lambda_i^n (c_i^n - |X-Y|)^{+} \right]
			\end{equation}
(using the monotone convergence theorem). Moreover, for any pair of random variables $(X,Y)$
			\begin{equation*}
			\lim_{n \to \infty} \sum_{i=1}^{n} \mathbb{E} \left[ \lambda_i^n (c_i^n - |X-Y|)^{+} \right]
			\quad\leq\quad \lim_{n \to \infty} \sum_{i=1}^{n} \sup_{Y: \mathcal{L}(Y) =\mathcal{L}(X + a)}  \mathbb{E} \left[ \lambda_i^n (c_i^n - |X-Y|)^{+} \right]
			\end{equation*}
			and hence
			\begin{equation}\label{eq:convexityAux}
			\begin{split}
&\sup_{Y: \mathcal{L}(Y) =\mathcal{L}(X + a)}  \lim_{n \to \infty} \sum_{i=1}^{n} \mathbb{E} \left[ \lambda_i^n (c_i^n - |X-Y|)^{+} \right]
\quad\leq\quad
\lim_{n \to \infty} \sum_{i=1}^{n} \sup_{Y: \mathcal{L}(Y) =\mathcal{L}(X + a)}  \mathbb{E} \left[ \lambda_i^n (c_i^n - |X-Y|)^{+} \right]
\\
&\quad=\quad \lim_{n \to \infty} \sum_{i=1}^{n}  \mathbb{E} \left[ \lambda_i^n (c_i^n - |X-\operatorname{AMR}(X)|)^{+} \right] 
\quad=\quad \mathbb{E} \left[ \varphi(|X-\operatorname{AMR}(X)|) \right]
\\
&\quad\leq\quad \sup_{Y: \mathcal{L}(Y) =\mathcal{L}(X + a)} \mathbb{E} \left[ \varphi(|X-Y|) \right] 
\quad=\quad \sup_{Y: \mathcal{L}(Y) =\mathcal{L}(X + a)} \lim_{n \to \infty} \sum_{i=1}^{n} \mathbb{E} \left[ \lambda_i^n (c_i^n - |X-Y|)^{+} \right] \,,
			\end{split} 
			\end{equation}
			where in the first equality above Lemma 
\ref{lem:AMRoptimality} is applied to a distribution function $F$ and its right 
shift $x \mapsto F(x-a)$, with the resulting optimal AMR coupling denoted by 
$(X,\operatorname{AMR}(X))$. This shows that
\begin{equation*}
\sup_{Y: \mathcal{L}(Y) =\mathcal{L}(X + a)}  \lim_{n \to \infty} \sum_{i=1}^{n} \mathbb{E} \left[ \lambda_i^n (c_i^n - |X-Y|)^{+} \right] 
\quad=\quad
\lim_{n \to \infty} \sum_{i=1}^{n} 
\sup_{Y: \mathcal{L}(Y) =\mathcal{L}(X + a)}  \mathbb{E} \left[ \lambda_i^n (c_i^n - |X-Y|)^{+} \right] \,.
\end{equation*}
			Combining this equality with Lemma \ref{lemmaSupIsConvex}, we conclude that the function
\begin{equation*}
			a \quad\mapsto\quad
\sup_{Y: \mathcal{L}(Y) = \mathcal{L}(X + a)} \lim_{n \to \infty} \sum_{i=1}^{n} \mathbb{E} \left[ \lambda_i^n (c_i^n - |X-Y|)^{+} \right]
\end{equation*}
			is both bounded, continuous, and also convex, being a limit of a sequence of sums 
of bounded continuous convex functions. It is also non-increasing, since 
any bounded continuous convex function on $[0, \infty)$ must be non-increasing. 
Combining this with \eqref{eq:mct} concludes the proof.
		\end{proof}	
	\end{lemma}

The representation \eqref{eq:mct} 
permits extension of the optimality result for AMR (Lemma 
\ref{lem:AMRoptimality}) to all bounded, continuous, convex, decreasing 
functions $\varphi$. Note further, that the optimization problem of maximizing 
the payoff $\mathbb{E}\left[ \varphi(|X-Y|) \right]$ for such $\varphi$ is 
equivalent to the problem of minimizing the cost 
$\mathbb{E}\left[\Phi(|X-Y|)\right]$ for bounded, continuous, concave, 
increasing functions $\Phi$, and hence the AMR coupling solves this 
minimization problem for such $\Phi$. 
This 
property of AMR
extends
to unbounded, continuous, concave, increasing costs $\Phi$, 
so long as it is assumed that the random variable $X$ in the optimization problem (and hence 
also $Y$) has finite first moment (or even a weaker condition 
$\mathbb{E}\left[\Phi(|X|)\right] < \infty$). Indeed, $\Phi$ specified above can 
be approximated by a sequence of bounded functions $\Phi^n$ (satisfying all the 
remaining properties of $\Phi$), for which AMR solves the corresponding 
optimization problem by the discussion above. Then, due to the assumption on the 
finite moment of $X$, the monotone convergence theorem can again be applied to obtain
\begin{equation*}
\mathbb{E} \left[\Phi (|X-Y|)\right] = \lim_{n \to \infty} \mathbb{E} \left[\Phi^n(|X-Y|)\right] \,,
\end{equation*}  
permitting the conclusion that AMR solves the optimization problem for $\Phi$:
\begin{lemma}\label{lem:AMRoptimality2}
	Let $F$ and $G$ be cumulative distribution functions such that there 
exists $\zeta \in \mathbb{R}$ so that the function $F - G$ is 
strictly increasing on $(-\infty,\zeta]$ and strictly decreasing on 
$[\zeta,\infty)$. Let $\Phi$ be a continuous, concave, increasing cost function 
(possibly unbounded). Finally, suppose $\mathbb{E}[\Phi(|X|)] < \infty$ if $X$ is a 
random variable with distribution function $F$. Then the AMR coupling 
$(X,Y)$ is the unique optimal coupling
(in the sense of Remark \ref{remarkUniqueness}) 
minimizing the payoff function 
$\mathbb{E}[\Phi(|X-Y|)]$. 
\end{lemma}
We stress that the unimodality of the distribution of $X$ is  
crucial for the existence of a single coupling optimal for all such costs 
(see Example \ref{exampleUnimodality}).

In the sequel we extend this optimality result to cover suitable 
stochastic processes, namely random walks and finite-activity L\'{e}vy 
processes. 

In the remainder of this section, let us focus on random walks with unimodal jumps.
 To this end, 
let $\phi$ be a convex, bounded, continuous, decreasing function as above and let $(X_k)_{k=0}^{\infty}$ be a Markov chain with a family of transition kernels 
$p = \{p(x,\cdot)\}_{x \in \mathbb{R}}$, where $p(x,A) = \mathbb{P}(X_n \in A|X_{n-1}=x)$ for all $n \geq 1$, i.e., the process $(X_k)_{k=0}^{\infty}$ is assumed to be time-homogeneous.
Following Section \ref{sec:unimodal}, we assume that each measure $p(x,\cdot)$ for any $x \in \mathbb{R}$ is a mixture of a Dirac mass at $0$ and a density which is strictly unimodal at $0$. For any $a > 0$, any $n \geq 1$ and any $l \in \{ 0, \ldots, n-1 \}$ we define
\begin{equation*}
\psi^{\phi}_{n,l}(a) \quad=\quad
\sup \mathbb{E} \left[ \phi(|X_n - Y_n|)
\;:\;
|X_l - Y_l| = a  
\right] \,,
\end{equation*}
where the supremum is taken over all 
immersion
couplings $(X_k, Y_k)_{k=l}^{\infty}$ of two copies of $(X_k)_{k=0}^{\infty}$ such that $|X_l - Y_l| = a$. Observe that, due to 
time-homogeneity and the fact that we work with immersion couplings, 
 $\psi^{\phi}_{n,l}(a)=\psi^{\phi}_{n-l,0}(a)=\psi^\phi_{n-l}(a)$ depends only on the value of the difference $n - l$ rather than on the individual values of $n$ and $l$. 
In other words, $\psi^{\phi}_{n,l}(a)$ is the supremum of the distance $\mathbb{E}\left[ \phi(|X_{n-l} - Y_{n-l}|) \right]$ taken over all 
immersion
couplings of Markov chains with the transition kernels $p$, such that $|X_0 - Y_0| = a$.

\begin{theorem}\label{thm:optimality}
	Let $(X_k)_{k=0}^{\infty}$ be a Markov chain with transition kernels $p$ as specified above. Let $(Z_k^{l,a})_{k=l}^{\infty}$ be a Markov chain such that $|Z_l^{l,a} - X_l| = a$ and 
	for all $k > l$ the random vector 
	$(X_k, Z_k^{l,a})$, 
	conditional on $(X_{k-1}, Z_{k-1}^{l,a})$, 
	is an AMR coupling of random variables defined as in \eqref{eq:skorokhod-def1}-\eqref{eq:skorokhod-construction}. Then, for any convex, bounded, continuous, decreasing function $\phi$, any $n \geq 1$ and any $l \in \{ 0, \ldots, n-1 \}$, 
	\begin{equation*}
	\psi^{\phi}_{n,l}(a) = \mathbb{E} \left[ \phi(|X_n - Z_n^{l,a}|) \right] \,.
	\end{equation*}
\end{theorem}	
\begin{proof}
By construction \((X, Z^{l,a})\) is an immersion coupling,
and indeed a Markovian coupling.

The proof uses mathematical induction
(following the ideas of dynamic programming and the Bellman principle): at level \(n\geq1\)
the inductive hypothesis is that
for any bounded convex 
continuous 
decreasing \(\phi\),
and for 
any 
\(m=1, \ldots, n\), the supremum
 \(\psi^\phi_m(a)\)
is realized using the iterated AMR coupling \((X,Z)\) with \(X_0-Z_0=a\):
moreover \(\psi^\phi_m\) itself is then also convex decreasing.

Consider an immersion coupling \((X,Y)\) such that $Y_0 - X_0 = a >0$
(the case \(a<0\) follows by exchanging rôles of \(X\) and \(Y\)).
Consider the first level \(n=1\) of the inductive hypothesis.
It follows from Lemmas \ref{lem:generalConvexity} and \ref{lem:AMRoptimality2}
that, for any bounded convex 
continuous 
decreasing \(\phi\),
the supremum \(\psi^\phi_1(a)\)
is realized using the single-jump AMR coupling \((X,Z)\) with \(X_0-Z_0=a\):
moreover \(\psi^\phi_1\) inherits the convex decreasing property of \(\phi=\psi^\phi_0\). Thus case \(n=1\) of the induction is valid.

Suppose the inductive hypothesis is valid at level \(n-1\).
To prove level \(n\) it suffices to show that
\[
\psi^\phi_{n}(a)\quad=\quad
\sup\{\mathbb{E}\left[|\phi(|X_n-Y_n|)|\right]\;;\;  |X_0-Y_0|=a \text{ and }(X,Y) \text{ is an immersion coupling}\}
\]
is maximized by taking \(Y=Z^{0,a}\), and that \(\psi^\phi_{n}\)
is convex decreasing.
For brevity we write \(Z^{0,a}=Z\).

We
argue as follows, First, use
 iterated conditional expectation and the definition of \(\psi^\phi_{n-1}\):
\[
\mathbb{E}\left[\phi(|X_n-Y_n|)\right] \quad=\quad
\mathbb{E}\left[\mathbb{E}\left[\phi(|X_n-Y_n|)\mid\mathcal{F}_1\right]\right]
\quad 
\leq 
\quad
\mathbb{E}\left[\psi^\phi_{n-1}(|X_1-Y_1|)\right]\,,
\]
where $\{ \mathcal{F}_n : n \geq 0 \}$ is the common filtration from the definition of $(X_n,Y_n)$ as an immersion coupling. The inequality above is a consequence of the definition of \(\psi^\phi_{n-1}\) and the fact that immersion coupling respects conditional expectations. 
Now employ the inductive hypothesis at level \(1\)
based on \(\psi^\phi_{n-1}\) rather than \(\phi\):
\[
 \mathbb{E}\left[\psi^\phi_{n-1}(|X_1-Y_1|)\right]\quad\leq\quad
\mathbb{E}\left[\psi^\phi_{n-1}(|X_1-Z_1|)\right]\,.
\]
Finally apply
iterated conditional expectation and the inductive hypothesis at level \(n-1\):
\[
\mathbb{E}\left[\psi^\phi_{n-1}(|X_1-Z_1|)\right]
\quad=\quad
\mathbb{E}\left[\mathbb{E}\left[\phi(|X_n-Z_n|\right]\mid\mathcal{F}_1\right]
\quad=\quad
\mathbb{E}\left[\phi(|X_n-Z_n|)\right]\,.
\]
Additionally
\[
\psi^\phi_n(|X_0-Z_0|)\quad=\quad
\mathbb{E}\left[\phi(|X_n-Z_n|)\right]
\quad=\quad\mathbb{E}\left[\psi^\phi_{n-1}(|X_1-Z_1|)\right]\,,
\]
so the convex decreasing property of \(\psi^\phi_{n-1}\) is inherited by \(\psi^\phi_n\).
This completes the proof of the inductive step, and so the result follows.
\end{proof}

By the equivalence of the maximization problem for convex payoffs and 
the minimization problem for concave costs, Theorem \ref{thm:optimality} solves 
the optimization problem described in Section \ref{sec:introduction} for Markov 
chains with unimodal transition probabilities (possibly with Dirac mass at the 
mode). 
An additional 
assumption of $\mathbb{E} \left[ \Phi(|X_k|) \right] < \infty$ for all $k \geq 
1$ is required for unbounded concave costs $\Phi$
(Lemma 
\ref{lem:AMRoptimality2}).

\subsection{Example: necessity of unimodality}

If unimodality fails then the optimal 
coupling may depend on the cost function.
\begin{example}\label{exampleUnimodality}[Two-point equiprobable distribution]
	Suppose \(X\) has a two-point equiprobable distribution with probability 
	masses of \(1/2\) each, located at \(0\) and \(1\).
	Suppose \(Y\) has the same distribution, but shifted \(\alpha\) units to the right for some $\alpha>0$, 
	thus with probability masses located at
	\(\alpha\) and \(1+\alpha\).
	Consider two extreme forms of coupling:
	\begin{enumerate}
		\item Synchronous coupling,
		based on the transportation plan sending \(0\) to \(\alpha\) and \(1\) to 
		\(1+\alpha\).
		\item Anti-monotonic rearrangement (which is equivalent here to a reflection coupling, since the distributions are symmetric), 
		based on the transportation plan sending \(0\) to \(1+\alpha\) and \(1\) to 
		\(\alpha\);
	\end{enumerate}
Evaluate these couplings using a concave cost based on 
	\(c_\gamma(x,y)=|x-y|^\gamma\), with \(\gamma\in(0,1)\). 
Consider the difference of the cost under plan 2 minus cost under plan 1 as a function of $\alpha$:
	\[
	f(\alpha) \quad=\quad
	(1+\alpha)^{\gamma} + |1-\alpha|^{\gamma} - 2 \alpha^{\gamma} \,.
	\]
It is an exercise to verify that the function $f(\alpha)$ is always positive for $\alpha \geq 1$, for all values of $\gamma \in (0,1)$. 
Hence, for $\alpha \geq 1$, the cost under plan 1 is less than the cost under 
plan 2 and hence the synchronous coupling is the preferred choice. On the other 
hand, numerical solution of the equation $f(\alpha) = 0$ for $\alpha \in (0,1)$ 
and $\gamma \in (0,1)$ shows that the optimal strategy is always 
reflection if 
	\(\alpha\leq0.7071\ldots\,\), whereas for \(0.7071\ldots<\alpha<1\) the optimal strategy depends on the choice of 
	\(\gamma\in(0,1)\).
\end{example}

\section{Optimal Markovian coupling for finite-activity L\'{e}vy processes}\label{sec:Levy}

This section employs a uniformization procedure for 
finite-activity L\'{e}vy processes (\emph{i.e.}, compound Poisson processes),
which generalizes the uniformization 
procedure for finite state-space continuous-time Markov chains
\cite[Section 6.7]{Ross-2014}. 
It permits representation
of any immersion coupling of finite-activity L\'{e}vy processes with a L\'{e}vy 
measure $\nu$ of mass $\nu(\mathbb{R})<\infty$, in terms of a Poisson process 
with intensity $2\nu(\mathbb{R})$ supplying jump times for both the coupled 
processes. This in turn permits extension of the optimality result from the 
previous section from Markov chains to finite-activity L\'{e}vy processes.

Let $(X_t, Y_t)_{t \geq 0}$ be an immersion coupling of two copies of a 
real-valued compound Poisson process with a L\'{e}vy 
measure $\nu$, 
adapted to a common filtration $(\mathcal{F}_{t})_{t \geq 0}$. 
 Then
there exist Poisson processes $(N_t^1)_{t \geq 0}$ and $(N_t^2)_{t \geq 0}$ 
 with intensity $\nu(\mathbb{R})$ and two sequences of i.i.d.\ random variables 
$(Z_i^1)_{i=1}^{\infty}$, $(Z_i^2)_{i=1}^{\infty}$ with 
$\mathcal{L}(Z_i^j)= 
\nu/\nu(\mathbb{R})$ 
for $i \in \mathbb{N}_{+}$, $j \in \{ 1, 2 \}$, respectively independent 
of $(N_t^1)_{t \geq 0}$ and $(N_t^2)_{t \geq 0}$, with 
\begin{equation}\label{eq:compoundPoissonCoupling}
X_t \quad=\quad X_0 + \sum_{i=1}^{N_t^1} Z_i^1 \qquad \text{ and } \qquad  Y_t \quad=\quad Y_0 + \sum_{i=1}^{N_t^2} Z_i^2 \,.
\end{equation}
Note that $(Z_i^1)_{i=1}^{\infty}$ and $(Z_i^2)_{i=1}^{\infty}$ 
are both i.i.d.\ sequences, but usually (except in the case of the independent 
coupling), the random variables $Z_i^1$ and $Z_i^2$ are dependent for any fixed 
$i \geq 1$.
Moreover,
\begin{equation*}
N_t^1 \quad=\quad \sum_{i=1}^{\infty} \mathbf{1}_{\{ \tau_i^1 \leq t \}} 
\qquad \text{ and } \qquad 
N_t^2 \quad=\quad \sum_{i=1}^{\infty} \mathbf{1}_{\{ \tau_i^2 \leq t \}} \,,
\end{equation*}
with sequences of random variables $(\tau^j_i)_{i=1}^{\infty}$ for $j \in \{ 1, 
2 \}$, such that $\mathcal{L}(\tau_i^j - \tau_{i-1}^j)=
\operatorname{Exp}(\nu(\mathbb{R}))$ for any $i \geq 2$, and all the random variables $\tau_i^j - 
\tau_{i-1}^j$ for $i \geq 2$ are mutually independent. We also write
\begin{equation*}
X_t \quad=\quad X_0 + \sum_{i=1}^{\infty} \mathbf{1}_{\{ \tau_i^1 \leq t \}} Z_i^1 
\qquad \text{ and } \qquad 
Y_t \quad=\quad Y_0 +  \sum_{i=1}^{\infty} \mathbf{1}_{\{ \tau_i^2 \leq t \}} Z_i^2 \,.
\end{equation*} 
If necessary,
we 
enrich the underlying common filtration	$(\mathcal{F}_{t})_{t \geq 0}$ to ensure the existence of a Poisson process \(\Psi\) of rate \(\nu(\mathbb{R})\)
and Uniform\((0,1)\) random variables \(M_1\), \(M_2\), \ldots, independent of each other and of the 
finite activity L\'evy processes, but appropriately adapted to the common filtration.

\begin{theorem}\label{thm:uniformization}
In the above situation, 
enriching the underlying common filtration as above 
if necessary, 
	 there exists a Poisson process $(N_t)_{t \geq 0}$ with intensity $2 
\nu(\mathbb{R})$ and two sequences of i.i.d.\ random variables 
$(\widetilde{Z}_i^1)_{i=1}^{\infty}$ and $(\widetilde{Z}_i^2)_{i=1}^{\infty}$ 
with 
$\mathcal{L}(\widetilde{Z}_i^j)=\frac{1}{2} \delta_0 + \frac{1}{2\nu(\mathbb{R})} 
\nu$ such that
	\begin{equation}\label{CPPcouplingRepresentation}
	X_t = X_0 + \sum_{i=1}^{N_t} \widetilde{Z}_i^1 \qquad \text{ and } \qquad Y_t = Y_0 + \sum_{i=1}^{N_t} \widetilde{Z}_i^2 \,.
	\end{equation}
	Moreover, both $(\widetilde{Z}_i^1)_{i=1}^{\infty}$ and $(\widetilde{Z}_i^2)_{i=1}^{\infty}$ are independent of $(N_t)_{t \geq 0}$,
	though again the random variables $\widetilde{Z}_i^1$ and $\widetilde{Z}_i^2$ are  typically dependent for any fixed 
$i \geq 1$.
\end{theorem}  
\begin{proof}
	We begin by constructing a suitable Poisson process $(N_t)_{t \geq 0}$ with intensity $2 \nu(\mathbb{R})$ given by
	\begin{equation*}
	N_t \quad=\quad \sum_{i=1}^{\infty} \mathbf{1}_{\{ \tau_i \leq t \}}\,,
	\end{equation*}
where the sequence $(\tau_i)_{i=1}^{\infty}$ is such that the random variables 
$\tau_i - \tau_{i-1}$  are mutually independent
with $\mathcal{L}(\tau_i - \tau_{i-1})=\operatorname{Exp}(2\nu(\mathbb{R}))$  
for all $i \geq 2$.
Moreover, 
it will be the case that
	\begin{equation}\label{eq:cppRepresentation}
	X_t \quad=\quad X_0 + \sum_{i=1}^{\infty} \mathbf{1}_{\{ \tau_i \leq t \}} Z_i^1 \mathbf{1}_{\{ \tau_i = \tau_i^1 \}} 
	\qquad \text{ and } \qquad
	Y_t \quad=\quad  Y_0 + \sum_{i=1}^{\infty} \mathbf{1}_{\{ \tau_i \leq t \}} Z_i^2 \mathbf{1}_{\{ \tau_i = \tau_i^2 \}} \,,
	\end{equation} 
so that $\widetilde{Z}_i^1 = Z_i^1 \mathbf{1}_{\{ \tau_i = \tau_i^1 \}}$ 
and $\widetilde{Z}_i^2 = Z_i^2 \mathbf{1}_{\{ \tau_i = \tau_i^2 \}}$ in the representation given by \eqref{CPPcouplingRepresentation}.

	Let $\Xi^i$ be Poisson point processes on $(0, \infty)$ counting the jumps 
of $(N^i_t)_{t \geq 0}$ for $i \in \{ 1, 2 \}$:
	\begin{equation*}
	\Xi^i(A) \quad=\quad \# \{ t \in A: \Delta N^i_t \neq 0 \} \,.
	\end{equation*}
Let $\lambda = \nu(\mathbb{R})$ be the common intensity of $(N^i_t)_{t \geq 0}$ for $i \in \{ 1, 2 \}$. 
Consider the point process $\Xi$ defined (as a counting process) by
	\begin{equation*}
	\Xi(A) \quad=\quad \# \{ t \in A: \Delta N^1_t \cdot \Delta N^2_t \neq 0 \} \,,
	\end{equation*}
counting the simultaneous jumps of $(N^1_t)_{t \geq 0}$ and $(N^2_t)_{t \geq 0}$.
The counting process $\Xi$ has an increasing c\`{a}dl\`{a}g 
compensator $(\Lambda_t)_{t \geq 0}$ (see e.g.\ Chapter 7 of 
\cite{DaleyVere-Jones-2003}). Since $\Xi_t \leq \Xi^1_t$ for all $t > 0$,
it follows that $\Lambda_t \leq \lambda t$ for all $t > 0$.

Hence the compensator $(\Lambda_t)_{t \geq 0}$ is absolutely continuous w.r.t.\ 
the Lebesgue measure on $\mathbb{R}_+$: so there exists a predictable 
process $(\eta_t)_{t \geq 0}$ with values in $[0,\lambda]$ such that $\Lambda_t 
= \int_0^t \eta_s \operatorname{d}{s}$.
	
Recall the independent Poisson point process $\Psi$ with intensity $\lambda$ on $(0,\infty)$,. 
Define a thinned version \(\widetilde{\Psi}\) of $\Psi$ as follows.
Suppose the points of \(\Psi\) are \(0<\widetilde{\tau}_1<\widetilde{\tau}_2<\ldots\),
deemed to be marked by the independent Uniform\((0,1)\) random variables \(M_1\), \(M_2\), \ldots.
Then set
	\begin{equation*}
	\widetilde{\Psi} \quad=\quad
	\left\{\widetilde{\tau}_i\;:\; i\in\{1,2,\ldots\} \text{ and } M_i \leq \eta_{\widetilde{\tau}_i}/\lambda\right\}
	\,.
	\end{equation*}
	Hence $\widetilde{\Psi}$ has the compensator $(\Lambda^{\widetilde{\Psi}}_t)_{t \geq 0}$ given by $\Lambda^{\widetilde{\Psi}}_t = \int_0^t \eta_s \operatorname{d}{s}$, i.e., $\Lambda^{\widetilde{\Psi}}_t = \Lambda_t$ for all $t \geq 0$.

Now define a counting process $\Upsilon$ by
	\begin{equation*}
	\Upsilon \quad=\quad \Xi^1 + \Xi^2 - \Xi + \widetilde{\Psi} \,.
	\end{equation*}
	Hence the compensator $(\Lambda_t^{\Upsilon})_{t \geq 0}$ of $\Upsilon$ is equal to
	\begin{equation*}
	\Lambda_t^{\Upsilon} \quad=\quad \lambda t + \int_0^t \left( \lambda - \eta_s \right) \operatorname{d}{s} + \int_0^t \eta_s \operatorname{d}{s} 
	\quad=\quad 2 \lambda t \,.
	\end{equation*}
Thus $\Upsilon$ is a Poisson process with intensity $2 
\lambda$, following the arguments of  \cite{Bremaud-1975} or  
\cite[Theorem 2.3]{Watanabe-1964}. Moreover, $\Upsilon$ can be viewed as constructed as a sum 
of two point processes: $\Xi^1$ and $(\Xi^2 - \Xi) + \widetilde{\Psi}$ which 
both have the compensator given by $\lambda t$ and hence are Poisson processes. 
Furthermore, $\Xi^1$ and $(\Xi^2 - \Xi) + \widetilde{\Psi}$ clearly do not have 
simultaneous jumps and hence they must be independent, using the arguments of 
\cite[Proposition XII-1.7]{RevuzYor-1999} or
\cite[Proposition 5.3]{ContTankov-2004}.
	
Set $N_t = \Upsilon([0,t])$. By construction this is a Poisson counting
process with intensity $2 \nu(\mathbb{R})$ for which 
\eqref{eq:cppRepresentation} holds with a sequence of $(\tau_i)_{i=1}^{\infty}$ 
obtained by counting the points of $\Upsilon$. Moreover, due to the independence of 
$\Xi^1$ and $(\Xi^2 - \Xi) + \widetilde{\Psi}$ (and, respectively, the 
independence of $\Xi^2$ and $(\Xi^1 - \Xi) + \widetilde{\Psi}$), it follows
that $\widetilde{Z}_j^1 = Z_j^1 \mathbf{1}_{\{ \tau_j = \tau_j^1 \}}$ and 
$\widetilde{Z}_j^2 = Z_j^2 \mathbf{1}_{\{ \tau_j = \tau_j^2 \}}$ for all $j \geq 
1$ have the distribution $\frac{1}{2}\delta_0 + \frac{1}{2\nu(\mathbb{R})} \nu$, 
since for any $u \in \mathbb{R}$
	\begin{equation}\label{eq:charFunctZtilde}
	\begin{split}
	&\mathbb{E}\left[ e^{i \langle u, \widetilde{Z}_j^k \rangle } \right] \quad=\quad
	\mathbb{E} \left[ e^{i \langle u, Z_j^k \mathbf{1}_{\{ \tau_j = \tau_j^k \}}\rangle }  \mathbf{1}_{\{\tau_j = \tau_j^k \}}\right] 
	+ \mathbb{E} \left[ e^{i \langle u, Z_j^k \mathbf{1}_{\{ \tau_j = \tau_j^k \}}\rangle }  \mathbf{1}_{\{\tau_j \neq \tau_j^k \}}\right] 
	\\
	&\quad=\quad \frac{1}{2}\mathbb{E} \left[ e^{i \langle u, Z_j^k \rangle }  \right] + \frac{1}{2} \mathbb{E} \left[ e^{i \langle u, 0\rangle }  \right] 
	\quad=\quad \frac{1}{2} \mathbb{E} \left[ e^{i \langle u, Z_j^k \rangle }\right] + \frac{1}{2} \,,
	\end{split}
	\end{equation}
	for all $j \geq 1$ and $k \in \{ 1 , 2 \}$, using the independence of 
$Z_j^k$ from both $\tau_j$ and $\tau_j^k$. Note that 
$\mathbb{P}(\tau_j = \tau_j^k) = \mathbb{P}(\tau_j \neq \tau_j^k) = 1/2$,
arising directly from the construction of 
$(N_t)_{t \geq 0}$.
	
	Note also that the mutual independence of $(\widetilde{Z}^1_j)_{j=1}^{\infty}$
	follows from the mutual independence of 
$(Z^1_j)_{j=1}^{\infty}$ combined with the mutual independence of the events $\{ 
\tau_j = \tau_j^1 \}$ for all $j \geq 1$. 
	
To show that the process $(N_t)_{t \geq 0}$ and the event $\{ \tau_j = \tau_j^1 \}$ are independent for any $j \geq 1$, 
we first show that $\widetilde{Z}_1^1 = Z_1^1 \mathbf{1}_{\{ \tau_1 = \tau_1^1 
\}}$ is independent of $\tau_1$. Then the assertion follows easily by induction. 
Note that for any $u_1$, $u_2 \in \mathbb{R}$,
\begin{equation}\label{eq:uniformizationAux}
	\begin{aligned}
	&\mathbb{E} \left[ e^{i\langle u_1, \tau_1 \rangle }   e^{i \langle u_2 , Z_1^1 \rangle}\right] 
	= \mathbb{E} \left[ e^{i\langle u_1, \tau_1 \rangle }   e^{i \langle u_2 , Z_1^1 \rangle} \mathbf{1}_{\{\tau_j = \tau_j^k \}} \right] 
+ \mathbb{E} \left[ e^{i\langle u_1, \tau_1 \rangle }   e^{i \langle u_2 , Z_1^1 \rangle}  \mathbf{1}_{\{\tau_j \neq \tau_j^k \}} \right] 
\\
	&\;=\; 2\times\mathbb{E} \left[ e^{i\langle u_1, \tau_1 \rangle }   e^{i \langle u_2 , Z_1^1 \rangle}  \mathbf{1}_{\{\tau_j = \tau_j^k \}} \right] \,,
	\end{aligned}
\end{equation}
	since $\mathbb{E} \left[ e^{i\langle u_1, \tau_1 \rangle }   e^{i \langle 
u_2 , Z_1^1 \rangle} \mathbf{1}_{\{\tau_j = \tau_j^k \}} \right] = \mathbb{E} 
\left[ e^{i\langle u_1, \tau_1 \rangle }   e^{i \langle u_2 , Z_1^1 \rangle} 
\mathbf{1}_{\{\tau_j \neq \tau_j^k \}} \right]$, following from the construction of 
$(N_t)_{t \geq 0}$ as a sum of two independent, identically distributed point 
processes. Hence for any $u_1$, $u_2 \in \mathbb{R}$ 
	\begin{equation*}
	\begin{split}
	&\mathbb{E} \left[ e^{i\langle u_1, \tau_1 \rangle }   e^{i \langle u_2 , \widetilde{Z}_1^1 \rangle}\right] 
	\quad=\quad \mathbb{E} \left[ e^{i\langle u_1, \tau_1 \rangle }   e^{i \langle u_2 , Z_1^1 \rangle} \mathbf{1}_{\{\tau_j = \tau_j^k \}} \right]
	+ \mathbb{E} \left[ e^{i\langle u_1, \tau_1 \rangle }   e^{i \langle u_2 , 0 \rangle} \mathbf{1}_{\{\tau_j \neq \tau_j^k \}} \right] 
	\\
	&\quad=\quad \frac{1}{2}\mathbb{E} \left[ e^{i\langle u_1, \tau_1 \rangle }   e^{i \langle u_2 , Z_1^1 \rangle} \right] + \mathbb{E} \left[ e^{i\langle u_1, \tau_1 \rangle } \mathbf{1}_{\{\tau_j \neq \tau_j^k \}} \right] 
	\\
	&\quad=\quad \frac{1}{2}\mathbb{E} \left[ e^{i\langle u_1, \tau_1 \rangle } \right] \mathbb{E} \left[ e^{i \langle u_2 , Z_1^1 \rangle} \right] + \frac{1}{2} \mathbb{E} \left[ e^{i\langle u_1, \tau_1 \rangle } \right] 
	\quad=\quad \mathbb{E} \left[ e^{i\langle u_1, \tau_1 \rangle } \right] \mathbb{E}\left[ e^{i \langle u, \widetilde{Z}_1^1 \rangle } \right] \,,
	\end{split}
	\end{equation*}
	where the second step uses \eqref{eq:uniformizationAux}, the third 
step uses the independence of $\tau_1$ and $Z_1^1$ for the first term and an 
argument analogous to \eqref{eq:uniformizationAux} for the second term, and 
the last step uses \eqref{eq:charFunctZtilde} for $j = k = 1$. This implies 
the independence of $\widetilde{Z}_1^1$ and $\tau_1$ and completes the proof.
\end{proof}

It is now possible to determine the optimal immersion coupling of two copies of a compound 
Poisson process with a L\'{e}vy measure $\nu$ having a strictly unimodal 
density. Based on Theorem \ref{thm:uniformization}, 
and possibly enriching the filtration using a marked Poisson process 
as in the statement of Theorem \ref{thm:uniformization},
any such 
coupling $(X_t, Y_t)_{t \geq 0}$ can be represented as in 
\eqref{CPPcouplingRepresentation}. This permits the interpretation 
of $(X_t, Y_t)_{t \geq 0}$ as a coupling of two processes always 
jumping simultaneously, albeit with a jump distribution which is different from the original, 
obtained by mixing the original with a Dirac mass at \(0\).

Given an immersion coupling $(X_t, Y_t)_{t \geq 0}$ as described above, with $|X_0 - Y_0| = a$,
define a Markov chain $(\bar{X}_k)_{k=0}^{\infty}$ derived from the jumps of $(X_t)_{t \geq 0}$ by requiring that
\begin{equation*}
\bar{X}_k \quad=\quad X_0 + \sum_{j=1}^{k} \widetilde{Z}_j^1
\end{equation*}
for $k \geq 0$. Based on $(\bar{X}_n)_{n=0}^{\infty}$, construct a Markov 
chain $(\bar{Z}_k^{0,a})_{k=0}^{\infty}$ starting from $Y_0$ which is 
step-by-step AMR-coupled to $(\bar{X}_n)_{n=0}^{\infty}$, as in the proof of 
Theorem \ref{thm:optimality}. 
Namely, set $\bar{Z}_0^{0,a} = Y_0$ (which 
implies $|\bar{X}_0 - \bar{Z}_0^{0,a}| = a$) and require that for any $k \geq 
1$ the random vector $(X_k - X_{k-1}, \bar{Z}_k^{0,a} - \bar{Z}_{k-1}^{0,a})$ is 
an AMR coupling defined as in 
\eqref{eq:skorokhod-def1}-\eqref{eq:skorokhod-construction}. 
Denote the jumps 
of $(\bar{Z}_k^{0,a})_{k=0}^{\infty}$ by
\begin{equation*}
\Delta \bar{Z}_k^{0,a} = \bar{Z}_k^{0,a} - \bar{Z}_{k-1}^{0,a}
\end{equation*}
for $k \geq 1$. Finally, 
also define a compound Poisson process $(Z_t^{0,a})_{t \geq 0}$ based on these 
jumps, with the same driving Poisson process as for $(X_t)_{t \geq 0}$, i.e.,
\begin{equation}\label{def:LevyAMRprocess}
Z_t^{0,a} \quad=\quad \bar{Z}_0^{0,a} + \sum_{k=1}^{N_t} \Delta \bar{Z}_k^{0,a} \,.
\end{equation}
Note that $(Z_t^{0,a})_{t \geq 0}$ is in fact constructed from $(X_t)_{t \geq 
0}$ by coupling all the jumps one-by-one using the AMR procedure. However, for 
our argument it is crucial that this procedure is applied to the extended jump 
distributions $\frac{1}{2\nu(\mathbb{R})}\nu + \frac{1}{2}\delta_0$ rather than 
just to $\frac{1}{\nu(\mathbb{R})}\nu$, cf.\ Figure \ref{figureAMRLevy}. From 
this construction it is the case that $|X_0 - Z_0^{0,a}| = a$ and, since the random 
variables $\Delta \bar{Z}_k^{0,a}$ and $\widetilde{Z}_k^2$ have the same law for 
each $k \geq 1$ (and $\bar{Z}_0^{0,a} = Y_0$), it follows that $(Z_t^{0,a})_{t \geq 
0}$ and $(Y_t)_{t \geq 0}$ have the same finite-dimensional distributions.

\begin{figure}
	\includegraphics{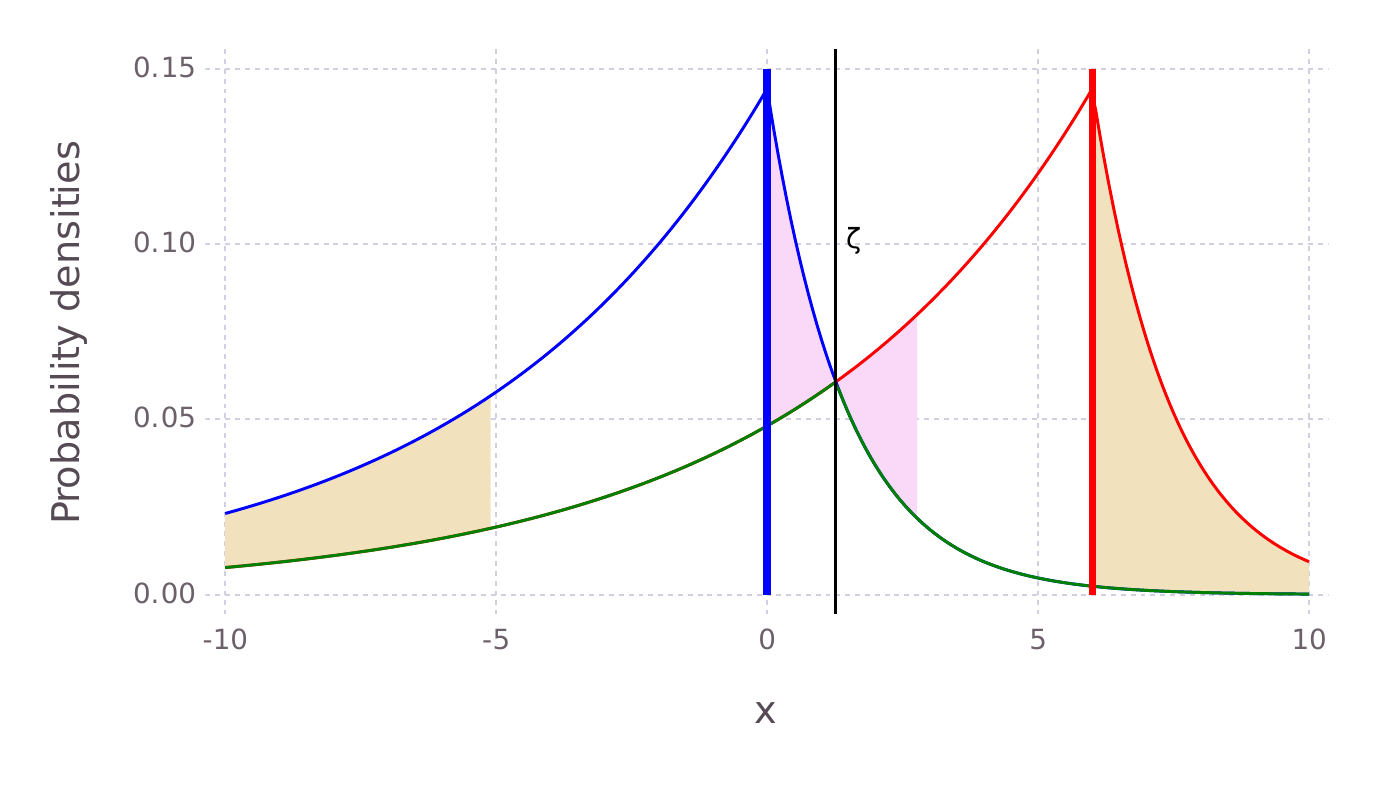}
	\caption{The anti-monotonic rearrangement in the case of a L\'{e}vy measure $\nu_1$ with finite support, having a non-symmetric unimodal density $f_1$ with added Dirac mass (represented by the bold line) at the mode (corresponding to the process $(X_t)_{t \geq 0}$) and its copy $\nu_2$ with density $f_2$ shifted to the right by $a > 0$ (corresponding to the process $(Z_t^{0,a})_{t \geq 0}$). The figure can be also interpreted as an illustration of the jump distribution of $(X_t)_{t \geq 0}$ (which is $\nu_1/\nu_1(\mathbb{R})$). Since the mass to the left of the mode of $f_1$ is larger than the mass to the right of the mode, some non-zero jumps of $(X_t)_{t \geq 0}$ to the left will be transformed into zero jumps of $(Z_t^{0,a})_{t \geq 0}$. 
		The underlying intuition is that when a unimodal probability density is mixed with
		a probability mass at the mode so that the mass at the mode is not exceeded by the mass of the resulting sub-probability density, then anti-monotonic rearrangement of the result with its translates never extends across the modes. Thus the white subset of the left region is transported to the right atom, while the white subset of the right region is transported to the left atom.
		\label{figureAMRLevy}}
\end{figure}

We can now prove the optimality result for the AMR coupling of finite-activity L\'{e}vy processes.
Note that it is expressed in terms of maximization of convex payoffs, 
but it equivalently solves the problem of minimization of concave costs, as 
explained in the discussion below the proof of Lemma \ref{lem:generalConvexity}.
\begin{theorem}\label{thm:LevyOptimality}
	Let $(X_t)_{t \geq 0}$ be a finite-activity L\'{e}vy process with a L\'{e}vy 
measure with strictly unimodal density and let $a > 0$. Let $(X_t, Z_t^{0,a})_{t 
\geq 0}$ be the AMR coupling defined by \eqref{def:LevyAMRprocess} and let 
$\phi:[0,\infty)\to[0,\infty)$ be a convex bounded continuous decreasing function.
Then it is the case that
	\begin{equation*}
	\sup \left\{\mathbb{E}\left[\phi(|X_t - Y_t|)\;;\; |X_0-Y_0|=a
	\text{ and } (X,Y) \text{ is an immersion coupling}\right] \right\}
	\,=\, \mathbb{E} \left[\phi(|X_t - Z_t^{0,a}|)\right] \,.
	\end{equation*} 
\end{theorem}
\begin{proof}For any convex, bounded, continuous, decreasing function $\phi$ and 
for any immersion coupling $(X_t, Y_t)_{t \geq 0}$ given by 
\eqref{CPPcouplingRepresentation},
	\begin{equation*}
	\begin{split}
	\mathbb{E} \left[\phi(|X_t - Y_t|)\right] 
	\quad&=\quad \sum_{n=0}^{\infty} \mathbb{E} \left[ \phi \left( \left| (X_0 - Y_0) + \sum_{j=1}^{N_t} (\widetilde{Z}_j^1 - \widetilde{Z}_j^2) \right| \right) \Bigg| N_t = n \right] \mathbb{P}[N_t = n] 
	\\
	\quad&=\quad \sum_{n=0}^{\infty} \mathbb{E} \left[\phi \left(\left| (X_0-Y_0) +  \sum_{j=1}^{n} (\widetilde{Z}_j^1 - \widetilde{Z}_j^2) \right| \right) \right]\mathbb{P}[N_t = n]
	\\
	\quad&\leq\quad \sum_{n=0}^{\infty} \mathbb{E} \left[\phi ( |\bar{X}_n - \bar{Z}^{0,a}_n| )\right] \mathbb{P}[N_t = n]
	\quad=\quad \mathbb{E} \left[\phi(|X_t - Z_t^{0,a}|)\right] \,,
	\end{split}
	\end{equation*} 
	where the inequality uses the optimality result for the discrete-time case 
(Theorem \ref{thm:optimality}). Taking the supremum on the left hand side over 
all immersion couplings $(X_t, Y_t)_{t \geq 0}$ as described above, it follows that 
that $(X_t, Z_t^{0,a})_{t \geq 0}$ is the optimal Markovian coupling.
\end{proof}

Finally, we stress that the optimal Markovian coupling for compound Poisson 
processes constructed above is unique. Indeed, if we take another optimal 
Markovian coupling $(\hat{X}_t, \hat{Y}_t)_{t \geq 0}$ given by
\begin{equation*}
\hat{X}_t = X_0 + \sum_{i=1}^{\hat{N}_t} \hat{Z}_i^1 \qquad \text{ and } \qquad \hat{Y}_t = Y_0 + \sum_{i=1}^{\hat{N}_t} \hat{Z}_i^2 \,,
\end{equation*}
with a Poisson process $(\hat{N}_t)_{t \geq 0}$ and i.i.d.\ sequences $(\hat{Z}_i^1)_{i=1}^{\infty}$ and  $(\hat{Z}_i^2)_{i=1}^{\infty}$ obtained in an analogous way as $(N_t)_{t \geq 0}$, $(\widetilde{Z}_i^1)_{i=1}^{\infty}$ and  $(\widetilde{Z}_i^2)_{i=1}^{\infty}$ above, respectively, then by the argument in the proof of Theorem  \ref{thm:LevyOptimality} and by the uniqueness of the optimal coupling for Markov chains, we conclude that for each $i \geq 1$ the random variables $\hat{Z}_i^1$ and $\hat{Z}_i^2$ have to be coupled via the AMR procedure. This gives uniqueness in law of the optimal Markovian coupling for continuous-time finite-activity L\'{e}vy processes.

To conclude, we briefly discuss two contrasting continuous-time examples.

\begin{example}
	Consider the case of a finite activity L\'{e}vy process with symmetric 
jumps. The above arguments show the optimal coupling is a ``reflection 
coupling'': the jumps of one process are the opposite of the jumps of the other 
till they both cross the midpoint between the starting points. At that crossover 
time the ``left'' process jumps to the location of the ``right'' process and 
coupling occurs. By arguments using the so-called recognition lemma \cite[Lemma 20]{ErnstKendallRobertsRosenthal-2019}, 
this coupling is maximal amongst all couplings, immersion or not.  
\end{example}

\begin{example}
	Consider the case of a finite activity L\'{e}vy process with exponential 
jumps (hence the jumps are always positive).
	The above arguments show that, in the optimal coupling with starting points 
$x<y$, the two coupled processes jump independently according to independent 
Poisson processes, till the first time the ``left'' process jumps past the 
``right'' process. At that crossover time the ``right'' process is forced to 
jump to the location of the ``left'' process. However, arguments using 
the so-called recognition lemma \cite[Lemma 20]{ErnstKendallRobertsRosenthal-2019} show that this coupling cannot be maximal.  
\end{example}

\noindent \textbf{Acknowledgements.} 
This is a theoretical research paper and, as such, no new data were created during this study. 
For the purpose of open access, the authors have applied a Creative Commons 
Attribution (CC BY) licence to any Author Accepted Manuscript version arising. 
The first author acknowledges support by the EPSRC grant EP/K013939. A substantial part of this work was completed while the second author was affiliated to the University of Warwick and supported by the EPSRC grant  EP/P003818/1. 
The third author was supported by the EPSRC grants EP/P003818/1 \& EP/V009478/1.
The first and third authors were also supported by The 
Alan Turing Institute under the EPSRC grant EP/N510129/1 and The Turing Institute Programme on Data-Centric Engineering funded by the Lloyd's Register Foundation.

\bibliographystyle{abbrv}
\bibliography{LevyCoupling_final}

\end{document}